\documentclass[a4paper,11pt, DIV12]{amsart}
\usepackage[english]{babel}
\usepackage[T1]{fontenc}
\usepackage[ansinew]{inputenc}
\usepackage{amsmath}
\usepackage{amsfonts}
\usepackage{ mathrsfs }
\usepackage{amssymb}
\usepackage{textcomp}
\usepackage[hidelinks]{hyperref}
\usepackage[arrow, matrix, curve]{xy}

\usepackage[a4paper, total={6in, 10in}]{geometry}

\usepackage{amsthm}										

\usepackage{enumerate}               

\usepackage{yhmath}

\usepackage{tikz}
\usepackage{tikz-cd}
\usetikzlibrary{matrix,arrows,decorations.pathmorphing}
\usepackage{cleveref}									

\newif\ifproofinsert
\proofinserttrue																			
\newsavebox\myproofbox
\ifproofinsert
\else
\renewenvironment{proof}{%
\begin{lrbox}{\myproofbox}\begin{minipage}{\linewidth}%
}{\end{minipage}\end{lrbox}}
\fi

\newtheorem{introthm}{Theorem}

\theoremstyle{definition}
	\newtheorem{defn}{Definition}[section]
	\newtheorem{bem}[defn]{Remark}
	\newtheorem{bsp}[defn]{Example}
	\newtheorem{const}[defn]{Construction}
\theoremstyle{plain}
	\newtheorem{satz}[defn]{Theorem}
	\newtheorem{kor}[defn]{Corollary}
	\newtheorem{lem}[defn]{Lemma}
	\newtheorem{prop}[defn]{Proposition}

\theoremstyle{plain}
  \newtheorem{Thm}[defn]{Theorem}
  \newtheorem{Lem}[defn]{Lemma}
  
  \newtheorem{Prop}[defn]{Proposition}

\theoremstyle{definition}
  \newtheorem{Def}[defn]{Definition}
	\newtheorem{D&P}[defn]{Definition \& Proposition}
  
	\newtheorem{Rem}[defn]{Remark}

\newcommand{\Z}{\mathbb{Z}}
\newcommand{\TT}{\mathbb{T}}

\newcommand{\R}{\mathbb{R}}

\newcommand{\T}{\mathbb{T}}

\newcommand{\Linear}{\mathbb{L}}

\newcommand{\vedge}{\land}
\newcommand{\del}{\partial}

\newcommand{\inj}{\hookrightarrow}

	\DeclareMathOperator{\PD}{PD}

	\DeclareMathOperator{\cone}{cone}

	\DeclareMathOperator{\trop}{trop}

	\DeclareMathOperator{\Hom}{Hom}

	\DeclareMathOperator{\sed}{sed}
	
	\DeclareMathOperator{\sing}{sing}
	\DeclareMathOperator{\supp}{supp}

	\DeclareMathOperator{\id}{id}

\newcommand{\inte}{\text{int}}

\DeclareMathOperator{\AS}{\mathcal{A}}

\DeclareMathOperator{\CS}{\mathcal{C}}
\DeclareMathOperator{\CC}{\mathcal{C}}

\DeclareMathOperator{\DS}{\mathcal{D}}

\DeclareMathOperator{\FS}{\mathcal{F}}

\DeclareMathOperator{\LS}{\mathcal{L}}

\DeclareMathOperator{\US}{\mathcal{U}}

\DeclareMathOperator{\RM}{\mathbb{R}}

\DeclareMathOperator{\TM}{\mathbb{T}}

\usepackage{amsmath}
\usepackage{amssymb}
\usepackage{hyperref}
\usepackage{babel}
\usepackage{pgf,tikz,pgfplots}
\usepackage{graphicx}

\usetikzlibrary{calc}
\usepackage{amsfonts}
\input xy
\xyoption{all}

\newtheorem*{ack}{Acknowledgements}

\tikzset{%
  add/.style args={#1 and #2}{to path={%
 ($(\tikztostart)!-#1!(\tikztotarget)$)--($(\tikztotarget)!-#2!(\tikztostart)$)%
  \tikztonodes}}
}

\title{Superforms, Tropical Cohomology, and Poincar\'e Duality}
\author{Philipp Jell, Kristin Shaw, Jascha Smacka}

\thanks{The second author's research  is supported by a postdoctoral research fellowship from the Alexander von Humboldt Foundation. 
The first and third author are respectively were partially supported by the collaborative research centre SFB 1085 "Higher Invariants" by the Deutsche Forschungsgemeinschaft.}

\address{Philipp Jell, Universit\"at Regensburg, Universit\"atsstra\ss e 31, 93053 Regensburg, Germany}
\email{philipp.jell@ur.de}
\address{Kristin Shaw, 
Technische Universit\"at
Berlin, MA 6-2, 10623 Berlin, Germany.}
\email{shaw@math.tu-berlin.de} 
\address{Jascha Smacka, Universit\"at Regensburg, Universit\"atsstra\ss e 31, 93053 Regensburg, Germany}
\email{jascha.smacka@ur.de}

\begin{document}

\maketitle 
\begin{abstract}
We establish a canonical isomorphism between two bigraded 
cohomology theories for polyhedral spaces:
Dolbeault cohomology of superforms
and  tropical 
cohomology.  
Furthermore, we prove Poincar\'e duality for cohomology of tropical manifolds, which are polyhedral spaces locally given by Bergman fans of matroids.  
\end{abstract}
\begin{center}
\today
\end{center}

\tableofcontents

\section{Introduction}

 Superforms  on $\R^r$  are bigraded real-valued differential forms introduced by Lagerberg \cite{Lagerberg}. They have differential operators $d'$, $d''$, and $d$ analogous to the differential operators $\del$, $\overline{\del}$, and $d$ on complex differential forms. 
 Recently, superforms restricted to tropicalizations  were  used by Chambert-Loir and Ducros to construct real-valued
differential forms on analytic spaces in the sense of Berkovich \cite{CLD}. Superforms have  also been used to provide a non-Archimedian analytic  description
of heights by Gubler and K\"unnemann  \cite{GublerKuennemann}.

A Poincar\'e lemma
 with respect to the differential operators $d'$ and $d''$ for superforms on polyhedral complexes in $\R^r$ and 
 Berkovich spaces  was proven by the first author \cite{Jell}. Here we consider the cohomology with respect to the operator $d''$. We call this the Dolbeault cohomology of superforms since the operator $d''$ behaves analogously to  the operator $\overline{\del}$ for complex 
differential forms. 

 Tropical 
 cohomology as introduced by Itenberg, Katzarkov, Mikhalkin, and Zharkov \cite{IKMZ}, is the cohomology of singular cochains 
 of a polyhedral complex 
  with non-constant coefficients. The coefficient systems are determined by the 
  geometry 
  of the complex  (see
Definition \ref{def:multitangent}).
Via the tropicalization  procedure, this cohomology theory can sometimes be related to the Hodge theory of projective varieties. For example, under suitable conditions on the tropicalization of a family of non-singular complex projective varieties, the dimensions of the tropical 
cohomology groups are equal to the Hodge numbers of a generic member of the family \cite[Corollary 2]{IKMZ}.

Our first goal is to prove that Dolbeault cohomology of superforms and tropical cohomology of a polyhedral space are canonically isomorphic. 
Before doing so,
we
extend the theory of superforms to polyhedral complexes contained in a partial compactification  $\T^r := [-\infty, \infty)^r$ of $\R^r$ which arises in tropical geometry. 
Superforms on 
 $\T^r$ require 
compatibility conditions along the strata of this partial compactification of $\RM^r$. 
We then extend the definition of superforms to polyhedral spaces, which  are topological spaces that are equipped with an atlas of charts to polyhedral complexes in $\T^r$ (see Definition \ref{def:polyhedralspace}).
For a polyhedral space $X$, we obtain complexes of sheaves of superforms $(\AS^{p,\bullet}_X, d'')$ by
 gluing spaces of superforms            										
   on open subsets.

   In Subsection \ref{sec:tropicalcohomology}, we recall the definition of the tropical cohomology groups $H^{p,q}_{\trop}(X)$ 
   of a  polyhedral space equipped with a face structure (see Definition \ref{def:facestructure}). 
   We also define the tropical cohomology groups with compact support $H^{p,q}_{\trop,c}(X)$. 
The Dolbeault cohomology of superforms is the cohomology of the complexes of global sections $(\mathcal{A}^{p, \bullet}_X(X), d'')$. We denote these groups by $   H^{p,q}_{d''}(X) = H^q((\mathcal{A}^{p, \bullet}_X(X), d''))$.
We also write $H^{p,q}_{d'',c}(X)$ for the cohomology of global sections with compact support (see Definition \ref{Def:DolbeautCohomology}).

The first theorem relates the Dolbeaut cohomology of superforms and tropical cohomology. 

\begin{introthm} \label{introthm1}
Let $X$ be a  polyhedral space 
 equipped with a face structure. Then there are canonical isomorphisms 
$$H^{p, q}_{\trop}(X) 
\cong H^{p,q}_{d''}(X) \qquad \text{and} \qquad 
H^{p, q}_{\trop,c}(X) 
\cong H^{p,q}_{d'',c}(X).$$
\end{introthm}

 To prove Theorem \ref{introthm1}, we first show that for every $p$, the complex $\AS_X^{p,\bullet}$ is an acyclic resolution of certain sheaves 
denoted 
 $\LS^p_X$ on $X$.
Tropical cohomology was already shown to be equivalent to the cohomology of constructible  sheaves, denoted
$\FS^p_X$  \cite[Proposition 2.8]{MikZhar}. Comparing explicit descriptions of these sheaves on a basis of the topology, we show that $\LS^p_X$ and $\FS^p_X$ are isomorphic, which implies the above theorem. In fact, the sheaves $\FS^p_X$ are defined for a polyhedral space $X$ even in the absence of a face structure. This relates the Dolbeault cohomology of superforms with the cohomology of the sheaves $\mathcal{F}^p_X$ for general polyhedral spaces (see Remark \ref{rem:Fpcohomology}).  
 
Secondly, we prove a version of Poincar\'e duality for tropical manifolds.  
For $X$ an $n$-dimensional
tropical space 
 (see Definition \ref{def:tropicalVariety}),  there is a map   $$\PD \colon H^{p,q}_{d''}(X) \rightarrow H^{n-p,n-q}_{d'',c}(X)^*,$$
 which we call the Poincar\'e duality map. This map is induced by integration of superforms (see Definition \ref{def:PoincareMap}), and thus is similar to the integration pairing on 
 the cohomology of a complex manifold. 
The fact that 
the Poincar\'e duality map on spaces of 
 superforms descends to cohomology when $X$ is a
tropical space  follows from an analogue  of Stokes' theorem  (see Theorem \ref{thm:StokesTropical}).

 Tropical manifolds are 
 tropical spaces
 with the extra condition that they are locally modeled on matroidal tropical cycles \cite{MikRau, Shaw:Thesis}. 
 A matroidal tropical cycle is supported on the Bergman fan 
of a  matroid and equipped with weight one. 
Some matroidal cycles  arise as tropicalizations of linear spaces, however they are much more general and may even have no algebraic counterpart \cite{Sturmfels:Poly}.
Despite perhaps being far from smooth objects in the  algebraic or differentiable sense, tropical manifolds exhibit many properties analogous to smooth spaces \cite{Shaw:Thesis}.
 Establishing Poincar\'e duality for the tropical cohomology of these spaces
provides another instance of this phenomenon. 

\begin{introthm}\label{introthm:Poincare}
If  $X$ is an $n$-dimensional tropical manifold then 
the Poincar\'e duality map 
is an isomorphism for all $p$ and $q$. 
\end{introthm}

As   in the proof of 
Poincar\'e duality for smooth manifolds, the statement is first established for the local models, 
which in our case are matroidal cycles. This is done in Propositions \ref{Prop:PDforMatroidalRr} and \ref{Prop:PDforMatroidaltimesT}.  
The main ingredient in the proof of the local case is a recursive description of matroidal cycles using tropical modifications 
(see Definition \ref{def:modificationsMat}). We restrict to tropical modifications of matroidal cycles which are  induced by deletion and contraction operations on the underlying matroids 
 \cite{Shaw:IntMat}. 
Poincar\'e duality for general tropical manifolds is then established from the local situation via standard methods.

In recent work, Adiprasito, Huh and, Katz consider an intersection ring associated to a matroid \cite{AdiprasitoHuhKatz}.  For a matroid  $M$,  the  graded ring $A^{*}(M)$   is shown to satisfy many striking properties in line with the cohomology rings of compact K\"ahler manifolds, such as Poincar\'e duality,  the Hard Lefschetz theorem, and an analogue of the  Hodge-Riemann bilinear relations. 
We expect that this ring is related to the cohomology groups presented here in the following way: For a matroid $M$ and $V$ its associated matroidal cycle,   there is a  suitable compactification $\overline{V}$ of $V$ 
for which $H^{k, k}(\overline{V}) \cong 
A^k(M) \otimes \mathbb{R}.$
Moreover, the product structures on  the tropical cohomology of $\overline{V}$ and $A^*(M)$ should also be isomorphic.

In addition to  the Poincar\'e duality relation established here, there is a lot of  interest in  other properties of the tropical cohomology groups. For instance, there are analogues of Lefschetz hyperplane section  theorems for tropical cohomology  \cite{AdiprasitoBjorner}.  
It was already shown that the tropical homology of tropical manifolds does not in general satisfy a direct translation of the Hodge-Riemann bilinear relations \cite{Shaw:11homol}. Furthermore, an  interesting open question is to establish the appropriate condition on a tropical manifold  $X$ so that $H^{p, q}(X) \cong H^{q,p}(X)$ \cite[Section 5]{MikZhar}.

It is also worthwhile to mention that tropical varieties can be used to construct currents on smooth complex projective varieties. 
This was recently used to  construct  a counter-example to the strongly positive Hodge conjecture \cite{BabaeeHuh}.  
Although their construction does not use the theory of superforms, it points to the power of connections between tropical geometry  and complex differential forms.

We now outline the presentation of this paper. 
Section \ref{chapter2} reviews  
 superforms on $\R^r$ and extends their definition to superforms on $\T^r$. 
For an open subset of the support of a polyhedral complex in $\T^r$ we define the space of $(p,q)$-superforms and show that this produces a sheaf. This construction  is also extended to produce sheaves of superforms on  polyhedral spaces. 
Section \ref{chapter3} recalls the definitions of tropical cohomology and calculates the cohomology of basic open sets (see Definition \ref{defn:basicopen}).
It also establishes a Poincar\'e lemma 
for 
the complexes of 
superforms on a polyhedral space and furthermore computes the sections of $\LS^p_X$ over basic open sets.  Following this, we show that the Dolbeault cohomology of superforms and tropical cohomology are  isomorphic (see Theorem \ref{thm:equivalenceDeRhamPQ}). 
 Subsection \ref{sec:integration} introduces integration  and proves 
 Stokes' theorem mentioned above. Finally, Subsection \ref{subsec:PD}  is devoted to the proof of Poincar\'e duality for tropical manifolds (see Theorem \ref{Thm:PoincareDualityII}).

 \begin{ack}
The authors would like to thank Walter Gubler, Johann Haas and Klaus K\"unnemann for comments on a preliminary draft, and also Karim Adiprasito, Grigory Mikhalkin, Johannes Rau and Ilia Zharkov for fruitful discussions. 
We are also grateful to two anonymous referees for helping us to improve this paper.

Furthermore, the authors would like to thank the Graduierten Kolleg "GRK 1692" by the Deutsche Forschungsgemeinschaft for making possible the lecture series by the second author that inspired this collaboration.
\end{ack}

\section{Superforms} \label{chapter2}

\subsection{Superforms on polyhedral subspaces of tropical affine space} \label{sec:formsOnComplexesTr}
In this subsection we define bigraded sheaves of 
 superforms on polyhedral complexes in tropical affine space $\T^r$. 
   We start by recalling the definitions for open subsets of $\RM^r$ due to Lagerberg \cite{Lagerberg}. After that we extend these to open subsets of $\T^r$ and to open subsets of polyhedral complexes in $\TM^r$.

\begin{defn}
Let $U \subset \R^r$ be an open subset. Denote by $\AS^q(U)$ the space of differential forms
 of degree $q$ on $U$. The space of \textit{$(p, q)$-superforms} on $U$ is defined as
\begin{align*}
\AS^{p,q}(U) := \AS^p(U) \otimes_{C^{\infty}(U)} \AS^q(U) =  \bigwedge^p {\R^r}^* \otimes_\R \AS^q(U),
\end{align*}
where $\bigwedge^p$ denotes the $p$-th exterior power. 
\end{defn}

If we choose a basis $x_1,\dots,x_r$ of $\R^r$, following \cite{CLD} and \cite{Gubler}, we  formally write a superform $\alpha \in \AS^{p,q}(U)$ as 
\begin{align*}
\alpha = \sum \limits _{|K| = p, |L| = q} \alpha_{KL} d'x_K \vedge d''x_L
\end{align*}
where $K = \{ i_1 , \dots i_p\} $ and $L = \{j_1,\dots j_q\}$ are ordered subsets of $\{1,\dots,r\}$, the coefficients $\alpha_{KL} \in C^\infty(U)$ are smooth functions and 
\begin{align*}
d'x_K \vedge d''x_L := (dx_{i_1} \vedge \dots \vedge dx_{i_p}) \otimes_\R (dx_{j_1} \vedge \dots \vedge dx_{j_q}). 
\end{align*}

There is a differential operator
\begin{align*}
d'' \colon \AS^{p,q}(U) =  \bigwedge^p {\R^r}^* \otimes_\R \AS^q(U) \rightarrow \AS^{p,q+1}(U) =  \bigwedge^p {\R^r}^* \otimes_\R \AS^{q+1}(U),
\end{align*}
given by $(-1)^p \text{id} \otimes D$, where $D$ is the usual 
differential operator on forms. In coordinates we have
\begin{align*}
d'' \left( \sum \limits_{K,L} \alpha_{KL} d'x_K \vedge d''x_L \right) &= \sum \limits_{K,L} \sum \limits _{i = 1} ^{r} \frac {\del \alpha_{KL} } {\del x_i} d''x_i \vedge d'x_K \vedge d''x_L \\
&:= (-1)^p \sum \limits_{K,L} \sum \limits _{i = 1} ^{r} \frac {\del \alpha_{KL} } {\del x_i} d'x_K \vedge  d''x_i \vedge d''x_L.
\end{align*}

\begin{bem}There are also differential operators $d' := D \otimes \id$ and $d := d' + d''$,
 which are  not considered in this paper. It is easy to see that the theories for $d'$ and $d''$ are symmetric 
up to sign. We choose to  consider the operator $d''$, since it produces the same cohomology as  tropical  cohomology. The cohomology of the operator $d'$ is isomorphic to that of $d''$ up to switching the bigrading.
\end{bem}

There is also a wedge product of superforms 
\begin{align*}
\wedge \colon \AS^{p,q}(U) \times \AS^{p',q'}(U) &\rightarrow \AS^{p+p',q+q'}(U) \\
(\alpha, \beta) &\mapsto \alpha \vedge \beta,
\end{align*}
which is, up to  sign, induced by the usual wedge product. In coordinates the wedge product  is given by 
\begin{align*}
(\alpha_{KL} d'x_K \vedge d''x_L ) \vedge ( \beta_{K'L'} d'x_{K'} \vedge d''x_{L'}) := &\alpha_{KL} \beta_{K'L'} d'x_{K} \vedge d''x_L \vedge d'x_{K'} \vedge d''x_{L'} \\
:= &(-1)^{p'q}  \alpha_{KL} \beta_{K'L'} d'x_{K} \vedge d'x_{K'} \vedge d''x_{L} \vedge d''x_{L'}.
\end{align*}
If one of $\alpha, \beta$ has compact support then so does $\alpha \wedge \beta$.
 Note that we have the usual Leibniz formula 
\begin{align*}
d''(\alpha \vedge \beta) = d'' \alpha \vedge \beta + (-1)^{p+q} \alpha \vedge d'' \beta.
\end{align*}

Let $\TT = [-\infty , \infty)$ and equip it with the topology of a half open interval. Then $\TT^r$ is equipped with the product topology. 
We write
$[r] := \{1, \dots , r\}$. 

\begin{defn} \label{defn:Tr}
The \emph{sedentarity} of a point $x \in \T^r$ is the subset $\sed(x) \subset [r]$ consisting of coordinates of $x$ which are $-\infty$. 

The space $\TT^r$ is naturally stratified by the sedentarity of points. For  $I \subset [r]$ set 
$$\R^r_I  := \{ x \in \TT^r \ | \ x_i = -\infty \text{ if and only if } i \in I \}.$$

Clearly, 
$\R^r_I \cong \R^{r-|I|}$. 
As a convention throughout, for a subset 
$S \subset \TT^r$ 
we  denote $S_I := S \cap \R^r_I$.  

 Moreover, for $J \subset I$ there is a canonical projection 
$\pi_{IJ}\colon \R^{r}_{J} \rightarrow \R^{r}_{I}$. Coordinate-wise the map $\pi_{IJ}$ sends  $x_i$ to $-\infty$ if $i \in I$ and to $x_i$ otherwise. 
\end{defn}

\begin{defn}\label{def:forms}
Let $U \subset \TT^r$ be an open subset. 
A \textit{$(p,q)$-superform $\alpha$ on $U$} is given by a collection of superforms $(\alpha_I)_{I \subset [r]}$  such that, 
 \begin{enumerate}
\item    
$\alpha_I  \in \AS^{p,q}(U_I)$ for all $I$, 
\item 
for each point $x \in U \subset \T^r$ of sedentarity $I$, there exists a neighborhood $U_x$ of $x$ contained in $U$ such that for each $J \subset I$ the projection satisfies $\pi_{IJ}(U_{x,J}) = U_{x,I}$ and  $\pi^*_{IJ}( \alpha_I|_{U_{x,I}}) = \alpha_J|_{U_{x,J}}$. 
\end{enumerate}

We denote the space of $(p,q)$-superforms on an open subset  $U$ by $\AS^{p,q}(U)$.
Note that a superform in $\AS^{0,0}(U)$ defines a collection of smooth functions on the subsets  $U_I$ which give a  continuous function on $U$. 
Therefore, we sometimes refer to $(0,0)$-superforms as smooth
 functions.

Condition $ii)$ of Definition \ref{def:forms} will be referred to as the \textit{condition of compatibility} of superforms along strata.  
Let $U \subset \T^r$ be an open subset and $\alpha \in \AS^{p, q}(U)$. Suppose that the points in $U$ have a unique maximal sedentarity and denote this by   $I$. 
 If for each $J \subset I$ we have $\pi_{IJ}^* \alpha_I = \alpha_J$, then we say that \textit {$\alpha$ is determined by $\alpha_I$ on $U$}. Notice that the condition of compatibility along strata implies that each $x \in U$ has an open neighborhood $U_x$ such that $\alpha|_{U_x}$ is determined by $(\alpha|_{U_x})_{\sed(x)}$ on $U_x$.

If $U \subset \T^r$ is an open subset and $\alpha  = (\alpha_I)_I \in \mathcal{A}^{p, q}(U)$ is a superform,  define 
$d'' \alpha$ to be given  by the collection $(d'' \alpha_I)_I$.
Pullbacks along the projections $\pi_{IJ}$ 
 commute with $d''$,
therefore $d''\alpha$ is a superform in 
$\AS^{p, q+1}(U)$. 
If also  $\beta  = (\beta_I)_I \in \AS^{p',q'}(U)$, then we define the wedge product $\alpha \vedge \beta := (\alpha_I \vedge \beta_I)_I \in \AS^{p+p', q+q'}(U)$.
This is indeed a superform on $U$, since the pullbacks along the projections commute with the wedge product.
\end{defn}

\begin{bem}
Notice that there is a natural isomorphism $J\colon \AS^{p, q}(U)\to\AS^{q, p}(U)$, which, up to sign, maps $d'x_K\otimes d''x_L$ to $d'x_L\otimes d''x_K$ 
 for $U \subset \T^r$ \cite[Section (1.2.5)]{CLD}. This is clear when $U \subset \R^r$. 
 When $U$ contains points of non-empty sedentarity the map $J$ preserves the condition of compatibility on the boundary strata. 
This involution is still well-defined for the spaces of superforms on polyhedral subspaces and polyhedral spaces defined in Subsection \ref{sec:polyhedralspace}.  
In the theory of  tropical  cohomology, outlined in Subsection \ref{sec:tropicalcohomology}, such an involution does not exist on the chain level.
\end{bem}

 \begin{bsp} \label{ex:part1}
 Consider an open
neighborhood $U$  of $-\infty \in \T$. For a $(p, q)$-superform $\alpha\in \AS^{p,q}(U)$ with $\max(p, q) = 1$,
by the condition of compatibility of superforms along strata,
 there must exist a smaller neighborhood $U' \subset U$ of $- \infty$ such that $\alpha$ is  zero on $U'$. 
 
 Similarly, a $(0, 0)$-superform on $U$ must be a constant function in some neighborhood of $-\infty$.
 \end{bsp}

In the next lemma we use upper indexing  of open sets to avoid confusion with the notation for the sedentarity.

\begin{lem} \label{lem:partitionsofunity}
Let $U \subset \TT^r$ be an open subset and $(U^l)_{l \in L}$ an open cover of $U$. Then there exist a countable, locally finite cover  $(V^k)_{k \in K}$ of $U$, a collection of non-negative smooth functions $(f^k\colon V^k \rightarrow \R)_{k \in K}$ with compact support, and a map $s \colon K \rightarrow L$ such that $V^k \subset U^{s(k)}$ for every $k\in K$, and $\sum \limits_{k \in K} f^k \equiv 1$.
\end{lem}

 Such a family of functions is called a \emph{partition of unity} subordinate to the cover $(U^l)_{l \in L}$.    

\begin{proof}
We first show that for any $x = (x_1, \dots , x_r)  \in \TT^r$ and any open neighborhood $x \in V$ there exists a non-negative function $f \in \AS^{0,0}(\TT^r)$  and a neighborhood $V'$ of $x$ such that $f|_{V'} \equiv 1$ and $\supp(f) \subset V$ is compact. This is clear if $r = 1$. Otherwise, a basis of open neighborhoods of $x$ is given by products of open sets in $\T$, thus we may assume $V$ to be of this form. Then taking functions $f^i$ on neighborhoods of $x_i$ 
in  $\T$ with the above property for every $i \in [r]$ 
and defining $f(x_1,\dots,x_r) = \prod f^i(x_i)$ gives the desired function. 

The general statement of the lemma now 
follows from standard arguments, see for instance
 the proof in \cite[Theorem 1.11]{Warner}. 
\end{proof}

\begin{defn}
A \textit{polyhedron} in $\R^r$ is a subset defined by a finite system of  affine
 (non-strict)  inequalities. A \textit{face} of a polyhedron $\sigma$ is a polyhedron which is obtained by turning  some of the defining inequalities of $\sigma$ into equalities. For conventions of convex geometry we follow \cite[Appendix A]{Gubler2}. 

A \textit{polyhedron} in $\TT^r$ is the closure of a polyhedron in  $\R^r_I \cong \R^{r - |I|} \subset \T^r$ for some $I\subset[r]$. A \textit{face} of a polyhedron $\sigma$ in $\T^r$ is the closure of a face of $\sigma \cap \R_J$ for some $J \subset [r]$.
A \textit{polyhedral complex} $\CC$ in $\TT^r$ is a finite set of polyhedra in $\TT^r$,
satisfying the following properties:
\begin{enumerate}
\item
For  a polyhedron $\sigma \in \CC$, if $\tau$ is a face of $\sigma$ (denoted $\tau \prec \sigma$) we have $\tau \in \CC$. 
\item
For two polyhedra $\sigma, \tau \in \CC$ the intersection  $\sigma \cap \tau$ is a face of both $\sigma $ and $\tau$. 
\end{enumerate}
The maximal polyhedra, with respect to inclusion, are called facets.
The support of a polyhedral complex $\CS$  is the union of all its polyhedra and is denoted $\vert \CS \vert$. If $X = |\CS|$, then $X$ is called a \textit{polyhedral subspace} of $\T^r$ and $\CS$ is called a \textit{polyhedral structure} on $X$.

The relative interior of a polyhedron $\sigma$ in $\T^r$  is denoted $\inte(\sigma)$. 
Given a polyhedral complex $\CS$ in $\T^r$ let $\CS_I$ denote the union of polyhedra $\sigma \in \CS$ for which $\inte(\sigma)$ is  contained in $\R^r_I$. 
By the definition of polyhedral complexes in $\T^r$, the collection $\CS_I$ is a polyhedral complex in $\R^r_I$. 
Notice that $|\CC_I| = |\CS|_I$. 
 For  a polyhedron  $\sigma$ in 
$\T^r$, we denote 
$\sigma \cap \R^r_I$  by $\sigma_I$.  
\end{defn}

\begin{defn} \label{defn:contraction}
Let $\CC$ be a polyhedral complex in $\T^r$ and $\sigma \in \CC$. Let $x \in \sigma$ be of sedentarity  $I$. 
Define \textit{the tangent space of $\sigma$ at $x$} to be $\Linear(\sigma,x) := \Linear(\sigma_I ) \subset \R^r_I$, 
where $\Linear(\sigma_I)$ is the tangent space to $\sigma_I$ at any point in its relative interior.

For  $U \subset \R^r$ an open subset  containing $x$, $\alpha \in \AS^{p,q}(U)$, and $s \in [p]$ 
\textit{the contraction of $\alpha$ by $v \in \R^r$ in the $s$-th component} is a $(p-1, q)$ superform denoted
$\langle \alpha; v \rangle_s \in \AS^{p-1,q}(U)$. The form  $\langle \alpha; v \rangle_s$ 
evaluated at a collection of vectors $v_1, \dots, v_{p-1}, w_1, \dots, w_q \in \Linear(\sigma,x)$  is
$$\langle \alpha(x); v_1, \dots, v, \dots ,v_{p-1}, w_1,\dots,w_q \rangle,$$ where the vector $v$ is in the $s$-th position.

For $U \subset \T^r$  an open subset and $\alpha = (\alpha_I)_I \in \AS^{p,q}(U)$,
 \textit{the contraction of $\alpha$ by $v \in \R^r$ in the $s$-th component} is the superform
$\langle \alpha, v \rangle_s \in \AS^{p-1, q}(U)$  given by the collection  $(\langle \alpha_I, \pi_{I, \emptyset}(v) \rangle_s)_I$.

Let  $U$ be an open subset containing $x$.  Then the evaluation of a superform $\alpha \in \AS^{p, q}(U)$ at a collection of vectors $v_1, \dots, v_p, w_1, \dots, w_q \in \Linear(\sigma,x)$ is denoted 
$\langle \alpha_I(x); v_1, \dots,v_p, w_1,\dots,w_q \rangle$. 

\end{defn}

Next we consider the restriction of bigraded superforms to polyhedral complexes in $\T^r$. 
\begin{defn}
Let $\CC$ be a polyhedral complex in $\TT^r$ and $\Omega \subset |\CC|$  an open subset. Then a \textit{$(p, q)$-superform on $\Omega$} is given by a superform $\alpha \in \AS^{p,q}(U)$ such that  $U\subset \TM^r$  is an open subset satisfying  $ \Omega =  U \cap |\CC|$.   Two such pairs $(U,\alpha)$ and $(U',\alpha')$
 are equivalent if for any $\sigma \in \CC$, any $x \in \Omega \cap \sigma$ of sedentarity $I$ and all tangent vectors $v_1,\dots,v_p,w_1,\dots,w_q \in \Linear(\sigma, x)$ we have 
\begin{align*}
\langle \alpha_I(x); v_1.\dots,v_p, w_1,\dots,w_q \rangle = \langle \alpha'_I(x); v_1.\dots,v_p, w_1,\dots,w_q \rangle.
\end{align*}
Let ${\AS}^{p, q}(\Omega)$ denote the set of equivalence classes of pairs $(U,\alpha)$ as above.
\end{defn}

\begin{bsp}\label{ex:tropicalline1}
Consider the standard tropical line $L \subset \R^2$. The space $L$ is the support of the one dimensional fan 
consisting of three rays in directions $(-1, 0), (0,-1)$ and $(1, 1)$.  
Let $\Omega$ be an open connected neighborhood of the origin in $L$. 
Since $L$ is one dimensional  $\AS^{p,q}(\Omega) = 0 $ if  $\max(p, q) > 1$. 
The space $\AS^{0,0}(\Omega)$ is the space of maps $f \colon \Omega \to \R$ 
which extend to a smooth function $\overline{f}: U \to \R$ for some open neighborhood $U$ of $\Omega $ in $\R^2$.

By construction $\AS^{p,q}(\Omega)$ is a $\AS^{0,0}(\Omega)$ module via the wedge product.
The space $\AS^{1,0}(\Omega)$ is spanned by $d'x$ and $d'y$ over $\AS^{0,0}(\Omega)$, where $x$ and $y$ are the coordinates on $\R^2$.
Note that these two forms each vanish on one of the rays of $L$ and agree on the ray in direction $(1, 1)$. The space $\AS^{0,1}(\Omega)$ is analogous. 

The space of superforms $\AS^{1,1}(\Omega)$ is  spanned by $$d'x \vedge d''x, d'x \vedge d''y, d'y \vedge d''x, \text{ and } d'y \vedge d''y.$$
The forms $d'x \vedge d''y$ and $d'y \vedge d''x$ both vanish on the half rays of $L$ that are in directions $(-1, 0)$ and  $(0, -1)$. 
On the ray in direction  $(1,1)$  we  have  $d'x \vedge d''y = d'y \vedge d''x$. This  shows that $d'x \vedge d''y = d'y \vedge d''x$ holds on $\Omega$. 
Furthermore,  we find that in the stalk of $\AS^{1, 1}$ at the vertex of $L$, the forms $d'x \vedge d''x$, $d'x \vedge d''y  = d'y \vedge d''x$, and $d'y \vedge d''y$ are linearly independent  over the  stalk of $\AS^{0,0}$ at the same point.
This differs from the situation over the complex numbers, where  the  space of top dimensional forms is always a free module of rank one over the space of smooth functions. 
\end{bsp}

\begin{bem}
By definition we have that $\alpha$ and $\alpha'$ define the same superform on $\Omega$ if and only if for all $I \subset [r]$ the superforms $\alpha_I$ and $\alpha'_I$ define the same superform on $\Omega_I$. 
Moreover, to determine if two superforms are equivalent when restricted to $\Omega$, by continuity,  it is enough to consider only points in the relative interior of facets.

The differential map $d''$ and the wedge product both descend to 
forms in  $\mathcal{A}^{p, q}(\Omega)$ 
in the sense that if superforms $\alpha, \beta \in \AS^{p,q}(\Omega)$ are given by $\alpha' \in \AS^{p,q}(U)$ and $\beta' \in \AS^{p,q}(U')$ then defining  $d'' \alpha$ to be given by $d'' \alpha'$ and $\alpha \vedge \beta$ to be  given by $\alpha'|_{U \cap U'} \vedge \beta'|_{U \cap U'}$ is independent of the choices of $\alpha'$ and $\beta'$. 
\end{bem}

For an open subset $\Omega$ of a polyhedral space $X \subset \T^r$,  the space of superforms $\AS^{p,q}(\Omega)$ does not depend on the underlying polyhedral complex $\CS$. To see this  we introduce the multi-(co)tangent spaces. These spaces will appear again in Section \ref{sec:tropicalcohomology} in relation to tropical (co)homology.

\begin{defn}\label{def:multitangentspacePoint}
Let $\CS$ be a polyhedral complex in $\T^r$ and $x \in |\CS_I|$. 
Then the \emph{p-th multi-tangent space} and \emph{multi-cotangent spaces at $x$} are respectively  
\begin{align*}
{\bf F}_p(x) =  \sum \limits_{ \tau \in \CS_I: x \in \tau} \bigwedge^p \Linear(\tau, x) \subset \bigwedge^p (\R^r_I) \qquad \text{and} \qquad
{\bf F}^p(x) = \left( \sum \limits_{ \tau \in \CS_I: x \in \tau} \bigwedge^p \Linear(\tau, x) \right)^*.
\end{align*}
\end{defn}

\begin{lem}
Let  $\Omega$ be an open subset of a polyhedral space $|\CS|  \subset \T^r$. Then the space of superforms $\AS^{p,q}(\Omega)$ only depends on $\Omega$.
\end{lem}

\begin{proof}
For  $x \in \Omega$, we claim that the vector space ${\bf F}_p(x)$ only depends on $\Omega$ and $x$. 
To see this, consider a refinement  $\CS'$ of the polyhedral complex  $\CS$. 
If $\sigma \in \CS$  and  $\sigma' \in \CS'$ are both maximal faces containing $x$ such that $\sigma'$ is contained in $\sigma$,  
then $\Linear(\sigma, x) = \Linear(\sigma', x)$. 
Furthermore, for each facet $\sigma \in \CS$ there exists at least one  $\sigma' \in \CS'$ with the above property. 
This shows that $\Linear(\sigma, x) = \sum_{\sigma' \in \CS', x \in \sigma' \subset \sigma} \Linear(\sigma', x)$ 
which in turn implies  that the definition of ${\bf F}_p(x)$ is the same for polyhedral structures $\CS$ and $\CS'$.  

Now given another polyhedral complex $\CS^{''}$ such that $\Omega$ is an open subset of $| \CS^{''}|$, we can find a polyhedral complex $\CS'$ which is a common refinement of both $\CS$ and $\CS^{''}$ when restricted to $\Omega$. 
It follows from the statement proved above that the vector space ${\bf F}_p(x)$ depends only on $x$ and $\Omega$. 

Now $\alpha \in \AS^{p,q}(\Omega)$  equals zero if and only if $\langle \alpha(x) ; v, w \rangle = 0$ for all $x \in \Omega, v \in {\bf F}_p(x),$  and $w \in {\bf F}_q(x)$. 
Finally, since  ${\bf F}_q(x)$ is independent of the polyhedral structure on $\Omega$ so is $\AS^{p,q}(\Omega)$. This completes the proof of the lemma. 
\end{proof}

For  a polyhedral subspace $X$ in $ \T^r$, the functor on open subsets of $X$ given by $\Omega \mapsto \AS_X^{p,q}(\Omega)$ will be denoted by $\AS_X^{p, q}$ or simply $\AS^{p,q}$ if the space $X$  is clear. The next lemma shows that this is an acyclic sheaf, where by acyclicity we always mean with respect to both the functor of global sections and the functor of global sections with compact support.

\begin{lem} \label{lem:sheaf}
For a polyhedral subspace $X$ in $\TT^r$, the presheaf
\begin{align*}
\Omega \mapsto \AS^{p,q}_X(\Omega)
\end{align*}
is a sheaf on $X$. 
Furthermore, this sheaf is fine, hence soft and acyclic. 
\begin{proof}
We start with the case $X = \TT^r$. In this case, all of the sheaf axioms are clearly satisfied  except for  the gluing property. Given a collection of superforms agreeing on intersections,
we can glue on each $\R^r_I$ getting a collection of superforms $\alpha_I$. The condition of compatibility  
along the boundary strata is  respected for the glued superforms since it is local and was respected for the superforms before gluing. 

For the general case we rely on the existence of partitions of unity.  Let $(\Omega^l)_{l \in L}$ be a collection of open sets and suppose that we have superforms $\alpha^l \in \AS^{p,q}(\Omega^l)$ which agree on the intersections $\Omega^l \cap \Omega^{l'}$ for $l, l' \in L$ and are the restrictions to $X$ of superforms $\beta^l \in \AS^{p,q}(U^l)$ for $\Omega^l = U^l \cap X$. We take a partition of unity $(f^k)_{k \in K}$ subordinate to the cover $(U^l)_{l \in L}$. By definition there is a map $s \colon K \to L$, so that if $s(k) = l$, then $f^k$ is supported on  $U_l$.  
Thus  $\beta = \sum \limits_{l \in L} \sum \limits_{k: s(k) = l} f^k \beta^l$ is a superform on the union $\cup_l \Omega^l$. Moreover for a fixed $l_0$ we have
\begin{align*}
\beta|_{\Omega^{l_0}} &= \sum \limits_{l \in L} \sum \limits_{k: s(k) = l} f^k|_{\Omega^{l_0}} \beta^l|_{\Omega^{l_0}} = \sum \limits_{l \in L} \sum \limits_{k: s(k) = l} f^k|_{\Omega^{l_0}} \alpha^l|_{\Omega^{l_0}} \\
&= \sum \limits_{l \in L} \sum \limits_{k: s(k) = l} f^k|_{\Omega^{l_0}} \alpha^{l_0} = (\sum \limits _{k \in K} f^k|_{\Omega^{l_0}}) \alpha^{l_0} = \alpha^{l_0}.
\end{align*}
 Therefore the superform given by $\beta$ restricted to $\cup_l \Omega^l$ gives the gluing of the superforms $\alpha^l$ above. This shows that $\AS_{X}^{p, q}$ is a sheaf on $X$. 
 
The fact that $\AS^{0,0}$ is fine follows from Lemma \ref{lem:partitionsofunity}. Then the  sheaves $\AS^{p,q}$ are also fine since  they are $\AS^{0,0}$-modules via the wedge product. Softness and acyclicity for global sections follows from  \cite[Chapter II, Proposition 3.5 \& Theorem 3.11]{Wells} respectively and acyclicity for sections with compact support follows from  \cite[III, Theorem 2.7]{Iversen}. 
\end{proof}
\end{lem}

\begin{defn}
Let $X$ be a polyhedral subspace of $\T^r$ and $\Omega$ an open subset. 
The \textit{support} of a superform $\alpha \in \AS^{p,q}(\Omega)$ is its support in the sense of sheaves, 
thus it consists of the points $x$ which do not have a neighborhood $\Omega_x$ such that $\alpha|_{\Omega_x} = 0$. 
The space of $(p, q)$-superforms with compact support on $\Omega$ is denoted $\AS^{p, q}_c(\Omega)$.
\end{defn}

\begin{lem} \label{lem:supportboundary}
Let $X$ be a polyhedral subspace of $\T^r$ and $\Omega$ an open subset. 
Let $\alpha = (\alpha_I)_{I} \in \AS^{p,q}(\Omega)$. 
Then we have $\supp \alpha=\bigcup_I \supp \alpha_I$. 
\end{lem}
\begin{proof}
Consider  $x \in \Omega_I$. 
If $x \notin \supp(\alpha_I)$, then there exists a neighborhood $U$ of $x$ in $\Omega_I$ such that $\alpha_I|_{U} = 0$. 
By the condition of compatibility, we may find a  neighborhood $V$ of $x$ in $X$ such that 
$\alpha|_V$ is determined by $\alpha|_{V_I}$ on $V$ and where $V_I \subset U$. 
Therefore,  $\alpha|_V = 0$.
This shows $\supp(\alpha) \subset  \bigcup_I \text{supp}(\alpha_I)$. 
The other inclusion is immediate, thus we have equality. 
\end{proof}

\subsection{Polyhedral spaces}\label{sec:polyhedralspace}

This subsection defines superforms on 
polyhedral spaces. These are spaces 
equipped with an atlas of charts to   polyhedral subspaces in $\TT^r$,  with coordinate changes given by extensions of affine maps. 
 First we establish pullbacks of superforms along extended affine maps, which permit the gluing of the sheaves $\AS^{p,q}$ defined in the last subsection.

Let $F\colon  \R^{r'} \rightarrow \R^r$ be an affine map and let $M_F$ denote the matrix representing the linear part of $F$. Let $I$ be the set of $i \in [r']$ such that the $i$-th column of 
$M_F$ 
 has only non-negative
entries. Then $F$  can be extended  to a map $$F\colon \left( \bigcup \limits_{J \subset I} \R^{r'}_J \right)\rightarrow \TT^r$$ by continuity, (equivalently,  using the usual $- \infty$-conventions for arithmetic). The extended map is also denoted by $F$. 

\begin{defn}
Let  $U' \subset \TT^{r'}$ be an open subset, then a map $F\colon U' \rightarrow \TT^r$, which is the restriction to $U'$ of a map arising as above is called an \textit{extended affine map}. Note that this only makes sense once we have $\sed(x) \subset I$  for all $x \in U'$. Similarly, for a polyhedral subspace $X'$ and an open subset $\Omega'$ of $X'$ an extended affine map $F\colon \Omega' \rightarrow \T^r$ is given by the restriction of an extended affine map to $\Omega'$. 
An extended affine map is called an \textit{integral extended affine map}, 
if it is the extension of an integer affine map  $\R^{r'} \to \R^r$, i.e.~its linear part is induced by a map of  the standard lattices 
$\Z^{r'} \to \Z^{r}$.
\end{defn}

\begin{defn}  [Pullback] \label{def:pullback} 
Let $U' \subset \TT^{r'}$ be an open subset and $F\colon U' \rightarrow \TT^r$ be an extended affine map. Let $U \subset \T^r$ be an open subset such that $F(U') \subset U$. Define 
\begin{align*} 
F\colon \{\text{sedentarities of points in } U'\} &\rightarrow 2^{[r]} \\ 
I' &\mapsto \sed(F(x)) \text{ for some and then every } x \in \R^{r'}_{I'}.
\end{align*}
Notice that this map respects inclusions.
$F$ induces an affine map $F_{I'}\colon \R^{r'}_{I'} \rightarrow \R^{r}_{F(I')}$ with $F_{I'}(U'_{I'}) \subset U_{F(I')}$. 
The \emph{pullback of the superform $\alpha = (\alpha_I)_I \in \AS^{p,q}(U)$  along $F$} is the collection of superforms $F^*(\alpha) := (F^*_{I'}(\alpha_{F(I')}))_{I'}$, 
where $F^*_{I'}(\alpha_{F(I')}) 
 \in \AS^{p, q}(U'_I)$. 
The next lemma shows that this collection satisfies the compatibility condition, and hence defines a superform on $U'$. Thus we have a pullback map $F^*\colon \AS^{p,q}(U) \rightarrow \AS^{p,q}(U')$.
\end{defn}

\begin{lem}\label{lem:pullback}
The pullback of a $(p,q)$-superform $\alpha$ on $U \subset \T^r$ 
along an extended affine map $F \colon U' \rightarrow U$ 
is a $(p, q)$-superform on $U' \subset \T^{r'}$. 
\end{lem}
\begin{proof}
We have to verify the condition of compatibility
of superforms along the strata. 
For $J' \subset I'$ we have $F(J') \subset F(I')$ and $F_{I'} \circ \pi_{I', J'} = \pi_{F(I')F(J')} \circ F_{J'}$. Thus if $\alpha \in \AS^{p, q}(U)$ is determined by $\alpha_{F(I')}$ on $U_x$, then we have 
\begin{align*}
\pi^*_{I'J'} (F^* (\alpha)_{I'}) &=\pi^*_{I'J'} F_{I'}^* (\alpha_{F(I')}) \\
&=  F_{J'}^* (\pi_{F(I')F(J')}^* (\alpha_{F(I')})) \\
&= F^*_{J'} (\alpha_{F(J')}) \\
&= (F^*(\alpha))_{J'},
\end{align*}
which shows that $F^*(\alpha)$ is determined by $F^*(\alpha)_{I'}$ on $F^{-1}(U_x)$. This shows the required compatibility.
\end{proof}

\begin{lem}\label{lem:pullbackC}  \label{chap4:lem:pullbackII}
Let $X \subset \T^r$ and  $X' \subset \T^{r'}$ be polyhedral subspaces and  let $\Omega \subset X$ and $\Omega' \subset X'$ be open subsets. If  $F \colon \Omega' \rightarrow \Omega$ is  an extended affine map, then  there exists a well defined pullback $F^*\colon  \AS^{p,q}(\Omega) \rightarrow \AS^{p,q}(\Omega')$, which is induced by the pullback in Definition \ref{def:pullback}. Moreover,  the pullback  is functorial and commutes with the differential $d''$ and the wedge product. 
\begin{proof}
Let $\alpha \in \AS^{p,q}(\Omega)$, then there exist open subsets $U' \subset \TT^{r'}$ and $U \subset \TT^r$ such that $\alpha$ is defined by some $\beta \in \AS^{p,q}(U)$, $F(U') \subset U$, and  $U' \cap X' = \Omega'$. Now the pullback $F^*(\beta) \in \AS^{p,q}(U')$ defines a superform on $\Omega'$. Set this to be $F^*(\alpha)$. To see that this is independent of the choice of $\beta$ we suppose that $\gamma$ is another superform on an open set
defining $\alpha$ on $\Omega$. After intersecting the domains of definition of $\beta$ and $\gamma$, we may assume that $\beta$ and $\gamma$ are defined on the same open set $U$. Since $\beta|_{\Omega} = \gamma|_{\Omega}$ we have that $\beta|_{\Omega_{F(I')}} = \gamma|_{\Omega'_{F(I')}}$ for all $I' \subset [r']$. Since the pullback via affine maps between vector spaces is well defined on polyhedral complexes  \cite[3.2]{Gubler}, we have $F_{I'}^*(\beta)|_{\Omega'_{I'}} = F_{I'}^*(\gamma)|_{\Omega'_{I'}}$ for all $I' \subset [r']$ and therefore $F^*(\beta)|_{\Omega'} = F^*(\gamma)|_{\Omega'}$, so that the pullback is well defined. 
 The last two statements of the lemma are direct consequences of the definition of pullbacks of forms along extended affine maps and the fact that the pullback by affine maps is functorial and commutes with $d''$ and the wedge product. 
\end{proof}
\end{lem}

We can now consider spaces equipped with an atlas of charts to polyhedral subspaces in $\T^r$. The following definition is a generalization of the definition of  tropical spaces given   in \cite{Mik:Applications, MikZhar, BIMS15}. We do not require our polyhedral subspaces to be rational, also the transition maps are required only to be extended affine maps, not integral affine. 
We also remove the finite type condition on the charts in \cite[Definition 1.2]{MikZhar}.

\begin{defn} \label{def:polyhedralspace}
A \textit{polyhedral space $X$} is a paracompact, second countable Hausdorff topological space with an atlas of charts $(\varphi_i \colon U_i \rightarrow \Omega_i \subset X_i)_{i \in I}$ such that:
\begin{enumerate}
\item
The $U_i$ are open subsets of $X$, the $\Omega_i$ are open subsets of polyhedral subspaces $X_i \subset \T^{r_i}$, and 
$\varphi_i \colon U_i \rightarrow \Omega_i$ is a homeomorphism for all $i$; 
\item
For all $i,j \in I$ the transition map
\begin{align*}
\varphi_i \circ \varphi^{-1}_j \colon \varphi_j(U_i \cap U_j) \rightarrow X_i 
\end{align*}
is an extended affine map. 
\end{enumerate}

\end{defn}

As in usual manifold theory, two atlases on $X$ are considered equivalent if their union is an atlas on $X$. 

The \textit{dimension} of $X$ is the maximal dimension among polyhedra which intersect the $\Omega_i$.
 The polyhedral complex is \textit{pure dimensional} if the dimension of the maximal, with respect to inclusion, polyhedra intersecting the open sets  $\Omega_i \subset X_i$ is
constant.

\begin{bsp} \label{ex:tropicalprojectivespace}
The tropical projective space $\TT P^r$ is the space 
\begin{align*}
(\T^{r+1} \setminus\{(-\infty,\dots,-\infty)\}) / \sim , \text{ where } x \sim y \text{ if there exists }
\lambda \in \R \text{ s.t. } x + (\lambda,\dots,\lambda) = y.
\end{align*}
For $i \in [r+1]$ the space $U_i = \{ [x] \in \T P^r \,\vert\, x_i \neq - \infty \}$ is homeomorphic to $\T^r$ via the maps
\begin{align*}
&\varphi_i\colon U_i \to \T^r; \; [x] \mapsto (x_j - x_i)_{j \in [r+1] \setminus \{i\}} \text{ and } \\
&\varphi_i^{-1}\colon \T^r \to U_i; \; x \mapsto (x_1,\dots,x_{i-1}, 0, x_{i+1}, \dots, x_r)
\end{align*}
The transition maps are given by 
\begin{align*}
\varphi_j \circ \varphi_i^{-1}\colon \T^r \setminus \R^r_{\{j\}} \to \T^r, \; x \mapsto (x_1 - x_j,\dots, x_{i-1} - x_j, -x_j ,x_{i+1}-x_j,\dots,x_r - x_j),
\end{align*}
which is an extended affine map. 
Thus $\T P^r$ together with the atlas $(\varphi_i\colon U_i \to \T^r)_{i \in [r+1]}$ is a polyhedral space.
\end{bsp}

\begin{defn}
Let $X$ be a polyhedral space with atlas $(\varphi_i\colon U_i \rightarrow \Omega_i \subset X_i)_{i \in I}$. 
The sheaf of superforms  $\AS^{p,q}_{U_i}$ is given by the pullback of the sheaves $\AS^{p,q}_{\Omega_i}$ via $\varphi_i$.
Then the  \textit{sheaf $\AS^{p,q}_{X}$ of $(p,q)$-superforms on $X$} is defined by  gluing of the sheaves $\AS^{p,q}_{U_i}$. 
The pullback of forms along the charts $\varphi_i$ is well defined and functorial, so this gives a well defined sheaf of superforms on $X$. We also again denote the sections with compact support by $\AS^{p,q}_c(X)$. 
\end{defn}

\begin{bsp} \label{ex:part2}
Let $[0,1]$ be the closed unit interval and  define the following charts: 
\begin{align*}
&\varphi_0: [0,1) \cong \TT^1  &&\varphi_1: (0,1]  \cong \TT^1 \\
&x \mapsto \tan ((x - 1/2) \pi)  &&x \mapsto -\tan ((x - 1/2) \pi).
\end{align*}
The interval $[0, 1]$ equipped with these two charts defines a polyhedral space, denoted by $X$. 
The single transition map for this atlas is
 $\varphi_0 \circ \varphi_1^{-1} \colon \R \rightarrow \R  = (x \mapsto -x)$. 
In Example \ref{ex:part1}, we saw that $(0,0)$-superforms on $\T$ are functions which are locally constant around $-\infty$. Thus $(0,0)$-superforms on $X$ are locally constant around both $0$ and $1$. Furthermore, similar to Example \ref{ex:part1}, superforms of positive degree vanish locally at  the two  boundary points of $X$. 
 
 The space $[0, 1]$ can also be equipped with an atlas consisting of a single chart which is just the inclusion 
$[0,1] \inj \R$. Denote this polyhedral space by $\tilde{X}$. Then superforms in $\AS^{0,0}(\tilde{X})$ are just smooth functions on $[0,1]$ in the usual sense, since superforms are not required to satisfy any compatibility conditions. Also the superforms $d'x$, $d''x$, and $d'x \vedge d''x$ are nowhere vanishing superforms of positive degree.  
\end{bsp}

\begin{Prop} \label{rem:complexonX}
Let $X$ be a polyhedral space, then the differential $d''$ and the wedge product of superforms on $X$ are well defined. 
Thus, for each $p \in \mathbb{N}$ we have a  complex
\begin{align*}
0 \rightarrow \AS^{p,0}_X \overset {d''} {\rightarrow} \AS^{p,1}_X \rightarrow \cdots. 
\end{align*}
If $X$ is $n$-dimensional, then $\AS_X^{p,q} =0$ for $\max(p,q) > n$. 
The sheaves $\AS^{p,q}_X$ are  
fine, hence soft and acyclic.
\end{Prop}

\begin{proof}
Let $(\varphi_i\colon U_i \rightarrow \Omega_i \subset X_i)_{i \in I}$ be an atlas for $X$. 
Thus $\AS^{p,q}_X = 0$ if $\max(p,q) > \dim(X)$, since $\AS^{p,q}_{\Omega_i} = 0$.  
Since $d''$ and the wedge product are compatible with pullbacks along extended affine maps, both of these maps  are well defined on the glued sheaves $\AS^{p,q}_X$. 
Since $X$ is paracompact and $\AS^{p,q}_X$ is glued from the fine sheaves $\AS^{p,q}_{\Omega_i}$, the sheaf $\AS^{p, q}_X$ is fine as well. 
Softness and acyclicity for global sections follows from  \cite[Chapter II, Proposition 3.5 \& Theorem 3.11]{Wells} respectively and acyclicity for sections with compact support follows from  \cite[III, Theorem 2.7]{Iversen}.
\end{proof}

\begin{defn}\label{Def:DolbeautCohomology}
Let $X$ be a polyhedral space, then the  \textit{Dolbeault cohomology} of superforms is defined as $H^{p,q}_{d''}(X) := H^q(\AS^{p,\bullet}_X(X), d'')$ and  the \textit{Dolbeault cohomology of superforms  with compact support} is defined as $H^{p,q}_{d'',c}(X) := H^q(\AS^{p,\bullet}_{X,c}(X), d'')$. 
\end{defn}

\section{Comparison of cohomologies}	\label{chapter3}

In this section we show that the Dolbeault cohomology of superforms agrees with tropical cohomology on polyhedral spaces. Subsection
 \ref{sec:tropicalcohomology}  recalls the definition of tropical cohomology using singular cochains. We then  give another description of tropical cohomology in terms of sheaves \cite{MikZhar}. In Subsection 
\ref{sec:Dolbeault}, we  show that Dolbeault cohomology of superforms is also equivalent to the cohomology of certain sheaves. We then calculate  sections of these sheaves and deduce from this that the sheaves defining tropical and Dolbeault cohomologies agree.

\subsection{Tropical cohomology}\label{sec:tropicalcohomology} 

This subsection describes tropical 
cohomology from \cite{IKMZ}. 
Recall the definitions of the multi-(co)tangent spaces from Definition \ref{def:multitangentspacePoint}. We now extend this definition to faces of a polyhedral complex. 

\begin{defn}\label{def:multitangent}
Let $\CC$ be a polyhedral complex  in $ \T^r$. 
Let $\sigma \in \CC$ and  $I \subset [r]$ be such that  $\inte(\sigma)  \subset \R^r_I$. 
The  $p$-th \emph{multi-tangent} and \emph{multi-cotangent space}
 of $\CC$ at $\sigma$ are the vector subspaces 
$${\bf F}_p(\sigma) =  \sum \limits_{ \tau \in \CC_I: \sigma \prec \tau} \bigwedge^p \Linear(\tau) \subset \bigwedge^p \R^r_I \qquad \text{and} \qquad
{\bf F}^p(\sigma) = \left( \sum \limits_{ \tau \in \CC_I: \sigma \prec \tau} \bigwedge^p \Linear(\tau) \right)^*,$$
respectively. 
\end{defn}

For $\sigma \prec \tau$ 
 for every $p$ there is a  map $i_{\tau \sigma} : {\bf F}_p(\tau) \to {\bf F}_p(\sigma)$, which is an inclusion of vector spaces if $\sigma $ and $\tau$ are of the same sedentarity.  
Otherwise, if  $\sigma$ is of sedentarity  $J$  and $\tau$ of sedentarity $I$, the map $i_{\tau \sigma}$  is given by the composition of the projection  $\pi_{IJ}$ and the above inclusion.   
 On the dual spaces ${\bf F}^p(\sigma)$,  the maps are reversed $r_{\tau \sigma}: {\bf F}^p(\sigma) \to {\bf F}^p(\tau)$. 
 
 To define the maps 
 $r_{\tau \sigma}: {\bf F}^p(\sigma) \to {\bf F}^p(\tau)$ for a polyhedral space $X$ we  impose an additional condition on $X$, which we call here a \emph{face structure} \cite[Definition 1.10]{MikZhar}.

\begin{defn}\label{def:facestructure} 
Let $X$ be a polyhedral space with atlas $ ( \varphi_i: U_i \to \Omega_i \subset  X_i)_{i\in I}$. 
A \textit{face structure} on $X$, consists of fixed polyhedral structures $\CS_i$ on $X_i$ for each $i$ and
 a finite number of closed sets $\{\overline{\sigma}_k\}$, called \textit{facets}, which cover $X$, such that
\begin{enumerate}
\item   each facet $\overline{\sigma}_k$ is contained in some chart $U_i$ for some $i$ 
such that $\varphi_i(\overline{\sigma}_k)$ is the intersection of $\Omega_i$ with a facet  $\tau_{ik}$ of the polyhedral complex $\CS_i$;
\item for any collection of facets $\mathcal{S}$ and $\overline{\sigma}_l \in \mathcal{S}$  the image of the intersection $\cap_{\overline{\sigma}_k \in \mathcal{S}} \overline{\sigma}_k$ in the chart $\varphi_i: U_i \to X_i$ containing $\overline{\sigma}_l$  is the intersection of $\Omega_i$ with a face of $\tau_{il}$. 
\end{enumerate} 

Given a face structure of $X$ a \emph{face} of $X$ is an intersection of facets. 
\end{defn}

Note that every open subset 
 of the support of a polyhedral complex in $\T^r$ is a polyhedral space with a face structure. For example, one can take  
  the facets to be the intersections of maximal polyhedra of the polyhedral complex with the open subset.

Two face structures on $X$ are \emph{equivalent}
 if there exists a common refinement, i.e.~a face structure $\{\overline{\sigma}_k \}$ on $X$ such that every facet $\overline{\sigma}_k$ is contained in a facet of each of the two original face structures.

Given a face structure on $X$, for faces $\sigma \prec \tau$ of $X$ there are  canonical maps 
 $i_{\tau \sigma}$ and $r_{\tau \sigma}$  between the multi-tangent and multi-cotangent spaces respectively.  
These maps are induced by the maps between the multi-tangent and multi-cotangent spaces of the images of the faces under a chart of the polyhedral space. 

\begin{bsp}\label{Exampel2.5}
We again consider the polyhedral space given by equipping the space $X = [0,1]$ with two charts to 
 $\T^1$ as in Example \ref{ex:part2}. Notice that $[0, 1]$ cannot be the only facet of a face structure on $X$ since it is not contained in a single chart. Choose in both charts the polyhedral structure on $\T$ with facets $[-\infty, 0]$ and $[0, \infty)$. Then we can take as a face structure on $X$ consisting of the  facets $\overline{\sigma}_1 = [0, 1/2]$ and $\overline{\sigma}_2 = [1/2, 1]$.
\end{bsp}

Denote by $\Delta_q$ the standard $q$-simplex. 

\begin{Def}\label{def:trophomol}
Let $X$ be a polyhedral space together with a face structure $\CS$ on $X$.
\begin{enumerate}
	\item
	For every face $\tau \in\CS$, we write $C_q(\tau)$ for the free $\R$-vector space 
	generated by continuous maps $\delta:\Delta_q\to\tau$ such that the image of $\inte(\Delta_q)$ is contained in $\inte(\tau)$ and  in addition
 	the image of each 
	open face of $\Delta_q$ is contained in the relative interior of a face of $\tau$.
	The space of \emph{tropical} $(p,q)$-\emph{chains  on} $X$ with respect to $\CS$ is 
	\[C_{p,q}(X):=\bigoplus_{\tau\in\CS}{\bf F}_p(\tau)\otimes C_q(\tau).\]

	\item
	
	For $\delta\in C_q(\tau)$ write $\partial\delta=\sum_{k=0}^{q}(-1)^{\epsilon_k} \delta^k$ for the usual boundary map, considered as a map $C_q(\tau)\to \bigoplus_{\sigma \prec \tau} C_{q-1}(\sigma)$.
For every $\sigma  \prec \tau$ in $\CS$ we have the map of multi-tangent
	 spaces, $i_{\tau \sigma}\colon {\bf F}_p(\tau)\to{\bf F}_p(\sigma)$.
	For $v\otimes\delta\in{\bf F}_p(\tau)\otimes C_q(\tau)$ we define the \emph{boundary operator} by 
    	\begin{align*}
	\partial(v\otimes\delta):=\sum_{k=0}^q(-1)^{\epsilon_k} v^k\otimes\delta^k\in C_{p,q-1}(X),
	\end{align*}
	where $v^k:=i_{ \tau \sigma}(v)$ when $\delta^k(\Delta_{q-1}) \subset \inte(\sigma)$. 
	We obtain complexes $(C_{p,\bullet}(X),\partial)$ of real vector spaces. 
	
		\item We define the \emph{tropical homology groups} to be
	\[H_{p,q}^{\trop}(X):=H_q(C_{p,\bullet}(X), \partial).\]
	Dually, we define \emph{tropical cochains} by $C^{p,q}(X):=\Hom(C_{p,q}(X),\RM)$ and the \emph{tropical cohomology} of $X$ as the cohomology of the dual complex
	\[H_{\trop}^{p,q}(X):=H^q(C^{p,\bullet}(X), \partial^*).\]
	
	\item
	We say that $\alpha\in C^{p,q}(X)$ has \emph{compact support} if there exists a compact subset $K_{\alpha} \subset X$ such that $\alpha(v\otimes\delta)\neq 0$ implies  
	$\delta(\Delta_q)\cap K_{\alpha} \neq\emptyset.$ 
	The cochains with compact support form a complex $C_c^{p,\bullet}(X)$ and we define \emph{tropical cohomology with compact support}
	by 
	\[H_{\trop,c}^{p,q}(X):=H^q(C_c^{p,\bullet}(X),\partial^*).\]
\end{enumerate}
\end{Def}
	
 \begin{bem}
	There are also cellular versions of tropical homology and cohomology \cite[Section 2.2]{MikZhar}. The advantage of the cellular  versions  is that they the (co)homology groups of finitely generated complexes. 
\end{bem}

Let $\CS$ be a polyhedral complex in $\T^r$. From the vector spaces ${\bf F}^p(\sigma)$,  it is possible to construct a sheaf on $|\CC| \subset \T^r$
following the lines of  \cite[Section 2.3]{MikZhar}. 
 For each open set $\Omega \subset |\CC|$, consider the  poset $P(\Omega)$ whose elements are the connected components $\sigma$ of faces of $\CC$ intersecting with $\Omega$. The elements of $P(\Omega)$  are ordered by inclusion and if $\sigma \prec \tau$ recall there are maps 
 $r_{\tau \sigma}: {\bf{F}}^p(\sigma) \to {\bf{F}}^p(\tau)$.
 
 \begin{defn}

For an open set $\Omega \subset |\CC|$ define the vector space
\begin{align*}
\mathcal{F}^p(\Omega)  := \varprojlim_{\sigma \in P(\Omega)} {\bf F}^p( \sigma).
\end{align*}

\end{defn}
The above defines a constructible 
sheaf 
of vector spaces 
on $|\CC|$ \cite{MikZhar}.  These sheaves do not depend on the polyhedral structure $\CS$ and thus are well defined for polyhedral subspaces. For a polyhedral space
 $X$, the sheaves $\mathcal{F}^p_X$ are defined by gluing along charts. Note that this definition does not require a face structure on $X$.

\begin{defn} \label{defn:basicopen}
A subset $\Delta \subset \T^r$  is an \textit{open cube} if it is a product of intervals which are either $(a_i, b_i)$ or $[-\infty, c_i)$ for $a_i \in \T$, $b_i,c_i \in \R \cup \{\infty \}$. 
 
 For a  polyhedral complex $\CC$ in $\T^r$, an open  subset $\Omega$ of $|\CC|$ is called a \textit{basic open} subset  if there exists an open cube $\Delta \subset \T^r$ such that $\Omega = |\CC| \cap \Delta$ and such that the set of polyhedra of $\CS$ intersecting $\Omega$ has a unique minimal element. Note that the sedentarity of the minimal polyhedron of $\Omega$ is the maximal sedentarity among points in $\Omega$.
 
Let $X$ be a polyhedral space with atlas $(\varphi_i \colon U_i \rightarrow \Omega_i \subset X_i)_{i \in I}$, such that for each $i$ we have a fixed  polyhedral structure $\CS_i$ on $X_i$. Then we say that an open subset $U$ is a \textit{basic open} subset (with respect to these structures) if there exists a chart $\varphi_i \colon U_i \rightarrow X_i$ such that $U \subset U_i$ and $\varphi_i(U)$ is a basic open subset of $|\CS_i|$.
\end{defn}

\begin{lem} \label{prop:basicopen}
Let $\CS$ be a polyhedral complex in $\T^r$. Then the basic open sets form a basis of the topology on $|\CC|$. Furthermore, if $\Omega$ is a basic open subset of  $|\CC|$ of sedentarity $I$, then $\Omega_I$ is a basic open subset of  $|\CC_I|$ 
in $\R^r_I$.
\end{lem}
\begin{proof}
Basic open sets form a basis of the topology of $|\CC|$ since open cubes form a basis of the topology of $\T^r$. For the second statement,  we have that $\Omega_I = |\CS_I| \cap \Delta_I$ and the minimal polyhedron of $\Omega_I$ is the same as the one of $\Omega$, so the lemma is proven. 
\end{proof}

\begin{lem} \label{lem:Fp}
Let $\CS$ be a polyhedral complex in $\T^r$ and $\Omega$ a basic open subset of $|\CS|$. Then 
\begin{align*}
\FS^p(\Omega) = \mathbf{F}^p(\sigma),
\end{align*}
where $\sigma$ is the minimal polyhedron of $\CC$. 
\end{lem}
\begin{proof}
Let $\Delta$ be an open cube such that $\Omega = \Delta \cap |\CS|$ and suppose that $I$ is such that $\inte(\sigma) \subset \R^r_I$. Then $\Omega \cap \sigma = \Delta \cap \sigma = (\Delta \cap \R^r_I) \cap \inte(\sigma)$ is connected, since it is the intersection of two convex sets. Thus the poset $P(\Omega)$ has $\Omega \cap \sigma$ as its unique minimal element and the lemma follows. 
\end{proof}

\begin{bsp}
Recall the definition of tropical projective space $\T P^r$ from Example \ref{ex:tropicalprojectivespace}.
 The sets $U_i$ are identified with $\T^r$ via the charts  $\varphi_i$. 
Since $\T^r$ is a basic open with minimal stratum $\sigma_\infty = \{(-\infty,\dots,-\infty)\}$, we
have $\FS^p(U_i) = \FS^p(\T^r) = { \bf F}^p(\sigma_\infty)$.
By definition we have that  ${\bf F}^p(\sigma_\infty) = 0$ for $p > 0$ and ${\bf F}^0(\sigma_\infty) = \R$. 
Therefore, $\FS^p(\T P^r) = 0$ for $p > 0$ and $\FS^0(\T P^r) = \R$.

Recall the  tropical line $L$ from Example \ref{ex:tropicalline1}. The entire line $L$ satisfies the conditions
to be a basic open subset. Its minimal polyhedron is the vertex, which we  denote by $\sigma$. 
By Lemma \ref{lem:Fp}, we have  $\FS^p(L) = {\bf F}^p(\sigma) $. 
Therefore, 
$$\FS^0(L) = \R \quad \text{and} \quad \FS^1(L) =  \langle -e_1, -e_2, e_1 + e_2 \rangle = \R^2.$$
An open  edge $\tau$ of $L$  is also a basic open
subset. 
If $v$ is the direction of $\tau$, then   $\FS^p(\tau) = \R$ for $p = 0$, $\FS^p(\tau) = \langle v \rangle$  for $p = 1$,  and $\FS^p(\tau) = 0$ otherwise. 

\end{bsp}

Next we compute $H^{p,q}_{\trop}(\Omega)$ for a basic open set $\Omega$.

\begin{Prop}\label{prop:acyclic}
Let $\Omega$ be a basic open subset of a polyhedral subspace $|\CS| \subset \T^r$, for a polyhedral complex $\CS$ in $\T^r$. 
Then 
\[H^{p, q}_{\trop}(\Omega) = 0\]
for $q>0$. Furthermore, we have canonical isomorphisms
\[H^{p,0}_{\trop}(\Omega) =  \mathbf{F}^p(\sigma), \]
where $\sigma$ is the minimal polyhedron of $\Omega$. 
\end{Prop}

\begin{proof}

Suppose first that the minimal polyhedron of $\Omega$ is of sedentarity $\emptyset$ and denote it by $\sigma$. 
Choose a point $x_0 \in \Omega$ which is in the relative interior of $\sigma$. 
By performing a translation of $\Omega$ we may assume that $x_0 = 0$. 

Since the complex  $C_\bullet^{\sing}(x_0, {\bf F}^p(\sigma))$ is a subcomplex of $C_{p,\bullet}^{\trop}(\Omega)$  there is a  canonical 
projection $\pi \colon C^{p,\bullet}_{\trop}(\Omega) \to C^\bullet_{\sing}(x_0, {\bf F}^p(\sigma))$. 
There is also a map 
$\iota \colon C^{\bullet}_{\sing}(x_0, {\bf F}^p(\sigma)) \to C^{p,\bullet}_{\trop}(\Omega)$
that is dual to the map which pushes forward the simplicies to $x_0$ and preserves the coefficients.

Define $f \colon \Omega \times [0,1] \to \Omega$ by $f(x,t) = (1-t)x$, and let  $f_t := f(\cdot, t)$. 
Notice that $f_1$ is the contraction of $\Omega$  to the origin. 
For all $t$ there is a map $f_t^* \colon C^{p, \bullet}_{\trop}(\Omega) \to  C^{p,\bullet}_{\trop}(\Omega)$ 
which is dual to the map which pushes forward the simplicies along $f_t$ and preserves the coefficients.
This is possible since $f$ preserves the polyhedral structure of $\Omega$. 
Notice that for $t = 1$,  we have $f_1^* = \iota \circ \pi \colon  C^{p,\bullet}_{\trop}(\Omega) \to  C^{p,\bullet}_{\trop}(\Omega)$. 
It is clear that $\pi \circ \iota = \id$. 
We claim that $\iota \circ \pi$ is homotopic to the identity.

Attached to $f$ there is a prism operator $P_f \colon C^{\sing}_{q-1}(\Omega) \to C^{\sing}_{q}(\Omega)$, which
provides a homotopy between $f_{0, *} = \id$ and $f_{1, *}$. 
Using the prism operator we can construct a map 
$P^{\vee}_f \colon C^{p,q}_{\trop}(\Omega) \to C^{p,q-1}_{\trop}(\Omega)$  on the tropical cochains groups  given by $(P^{\vee}_f)(\alpha)(v \otimes \delta) = \alpha(v \otimes P_f(\delta))$. 
It can be checked by following the argument for the case of constant coefficients 
that $P^{\vee}_f$ provides a homotopy between $\id$ and $f^*_1 = \iota \circ \pi$. 
Therefore, $H^{p, q}_{\trop}(\Omega) \cong H^q_{\sing}(x_0, {\bf{F}}_p(\sigma))$
and the statement of the proposition follows.

When the minimal polyhedron of $\Omega$ is of sedentarity $I \neq \emptyset$ Lemma \ref{lem:retractionSed}  
constructs a deformation retraction of $\Omega$ onto 
$\Omega_I$
 which preserves the underlying  polyhedral structure. 
Using this retraction we can apply the argument above and obtain the canonical 
isomorphism in the claim. 
This proves the proposition.
\end{proof}

\begin{lem}\label{lem:retractionSed} 
Let $\CS$ be a polyhedral complex and $\Omega$ a basic open subset with maximal sedentarity $I$. 
Then there exists a continuous map
\begin{align*}
g\colon \Omega \times [0,1] \to \Omega 
\end{align*}
such that 
\begin{enumerate}
\item 
$g(x, 0) = x$ for all $x \in \Omega$;
\item
$g(x, t) = x$ for all $x \in \Omega_I$ and all $t \in [0,1]$; 
\item
$g(x,1) \in \Omega_I$ for all $x \in \Omega$; 
\item
For $\sigma \in \CS$ and $x \in \inte(\sigma)$ we have $g(x, t) \in \inte (\sigma)$ for all $t \in [0,1)$. 
\end{enumerate}
\end{lem}

\begin{proof}
We will define a 
deformation retraction 
\begin{align*}
g(x,t) = x - \log(1- t) \cdot w(x)
\end{align*}
where 
 $w\colon \Omega \to \R^r$ is a continuous map. Notice that then property $i)$ from the statement of the lemma is satisfied. 

The map $w$ will be constructed so that for $x \in \inte(\sigma)$ and of sedentarity $J$ we have 
\begin{enumerate}
\item[2)]$w(x) = 0$ if $x \in \Omega_I$;
\item[3)] $w(x) \in S_{I \setminus J} := \{ y \in \R^r  \ \vert \ y_i < 0 \text{ if } i \in I \setminus J \text{ and } y_i = 0 \text{ else} \};$ 
        \item[4)] $w(x) \in \cone(\sigma) :=  \{ v \in \R^r \ \vert \  v_i = 0 \text{ for } i \in J \text{ and } y + n \cdot v \in \sigma \text{ for all } n \in \mathbb{N} \text{ and all } y \in \sigma\}$.
\end{enumerate}

Then property $2)$ of $w$ implies property $ii)$ in the statement of the lemma.  Also property $4)$  implies $iv)$. 
Property $3)$ implies that $g(x, 1)$ is of sedentarity $I$ for all $x \in \Omega$, and that for all $t$ the image of $g(x,t)$ is in a cube $\Delta$ defining $\Omega$ from Definition \ref{defn:basicopen}. This combined with $iv)$, shows that $g(x, 1) \in \Omega_I$. This proves property  $iii)$.

The map $w$ is constructed  inductively. 
If $x$ is of  sedentarity $I$, then we define $w(x) = 0$.
Given a polyhedron $\sigma \in \CS$ of dimension $k$, suppose 
we have already constructed $w$ on  the support of the $(k-1)$-skeleton of $\CS$ intersected with $\Omega$ .   
In particular, the map  $w$  is constructed on the boundary  $\partial \sigma$ intersected with $\Omega$. 
Properties $3)$ and $4)$ imply that the image of $\partial \sigma \cap \Omega$ under $w$ is contained in 
$ \partial ( \cone( \sigma) \cap S_{I \setminus \sed(\sigma) } ).$
Now since $ \cone( \sigma) \cap S_{I \setminus \sed(\sigma)}$ is 
convex, we can interpolate between the  values of $w$ on the boundary of $\sigma$ and define $w$ on all of $\sigma$.
This completes the proof of the lemma. 
\end{proof}

Next we relate the cohomology of the sheaf $\FS^p$ to tropical cohomology. Our arguments follow those presented in detail in  \cite[p.~110-113]{Ramanan} for  the case of the constant sheaf. 

\begin{defn}
Let $X$ be a polyhedral space which has a face structure. 
Then we define the sheaf $\underline{C}^{p,q}$ on $X$ to be the sheaf associated with
\begin{align*}
\Omega \mapsto C^{p,q}(\Omega).
\end{align*}
\end{defn}

\begin{lem} \label{Cpqsheaf}
The sheaves $\underline{C}^{p,q}$ are flasque. 
Furthermore, the canonical maps $(C^{p,q}(\Omega), \del) \rightarrow (\underline{C}^{p,q}(\Omega), \del)$ and 
$(C^{p,q}_c(\Omega), \del) \rightarrow (\underline{C}^{p,q}_c(\Omega), \del)$ are quasi-isomorphisms.
\end{lem} 
\begin{proof}

The presheaves $\Omega \mapsto C^{p,q}(\Omega)$ are flasque. 
Given an open subset $U$, an open cover $(U_i)_{i \in I}$ and elements $c_i \in C^{p,q}(U_i)$ which agree when 
restricted to intersections, we can define a cochain $c \in C^{p,q}(U)$ by assiging to every element $v \otimes \delta \in C_{p,q}(U)$ 
the value 
$c (v \otimes \delta) := c_i(v \otimes \delta)$ if $\delta(\Delta) \subset U_i$ and $c(v \otimes \delta) := 0$ else.
Thus the presheaves $C^{p,q}(\Omega)$ satisfy the glueing axiom. 
This imples that the map 
$C^{p,q}(\Omega) \rightarrow \underline{C}^{p,q}(\Omega)$
is surjective for all $\Omega$. 
This implies flasqueness of $\underline{C}^{p,q}$.

As in \cite[p.110]{Ramanan} for a $q$-simplex $\delta$, let $\tilde{b}(\delta)$ denote the collection of simplicies which is the  barycentric subdivision of $\delta$. 
For $\alpha \in C^{p,q}(\Omega)$ define the cochain $b(\alpha)$ by $b(\alpha)(v \otimes \delta) =  \alpha (v \otimes \tilde{b}(\delta))$. 
For an open cover $\US = (U_i)$ of $\Omega$ we define $C^{p,\bullet}(\Omega)^{\US}$ to be the subcomplex of $C^{p,\bullet}(\Omega)$ 
given by  cochains supported on simplicies which are contained in one of the  open sets $U_i$. 
As in \cite[Chapter 4, Proposition 4.10 i)]{Ramanan}, we can use $b$ to give  a  canonical map 
$C^{p,\bullet}(\Omega) \to C^{p,\bullet}(\Omega)^{\US}$
which is  a quasi-isomorphism. 
Following the proof of \cite[Chapter 4, Proposition 4.10 ii)]{Ramanan}, the claim follows from the fact that  
the map 
$C^{p,\bullet}(\Omega) \to \underline{C}^{p,\bullet}(\Omega)$
factors through $C^{p,\bullet}(\Omega)^{\US}$ for any cover $\US$ and
also that a cochain $\alpha \in C^{p,q}(\Omega)$ which vanishes in $\underline{C}^{p,\bullet}(\Omega)$ must also vanish in $C^{p,\bullet}(\Omega)^{\US}$ for some cover $\US$. 

Since the operator $b$ does not change the support of a cochain, the same arguments work when considering cochains with compact support.
\end{proof}

The next proposition is stated in \cite[Proposition 2.8]{MikZhar}. Here we provide the details of its proof.

\begin{prop} \label{Prop:tropsheaf}
For a polyhedral space $X$ equipped with a  
 face structure, there are canonical isomorphisms
\begin{align*}
H^{p,q}_{\trop}(X) \cong H^q(X, \FS^p) \text{ and } H^{p,q}_{\trop, c}(X) \cong H^q_c(X, \FS^p).
\end{align*}
\end{prop}

\begin{proof}
Once again we adopt the proof for constant coefficients to our situation.
Let $\Omega$ be a basic open subset of $X$. 
Then 
$0 \rightarrow \FS^p(\Omega) \rightarrow C^{p,0}(\Omega) \rightarrow \dots$ 
is exact by Lemma \ref{lem:Fp} and Proposition \ref{prop:acyclic}. 
Thus by Lemma \ref{Cpqsheaf} the complex 
$0 \rightarrow \FS^p(\Omega) \rightarrow \underline{C}^{p,0}(\Omega) \rightarrow \dots$
is also exact. 
Since basic open subsets form a basis of the topology, this means that 
$0 \rightarrow \FS^p \rightarrow \underline{C}^{p,0}\rightarrow \dots$
is an exact sequence of sheaves, and therefore an acyclic resolution of $\FS^p$ by Lemma \ref{Cpqsheaf}.
Thus, 
$$H^{q}(X, \FS^p) = H^q( \underline{C}^{p,\bullet}(X), \del) = H^q(C^{p,\bullet}(X), \del) = H^{p,q}_{\trop}(X)  $$
and
$$H^{q}_c(X, \FS^p) = H^q( \underline{C}^{p,\bullet}_c(X), \del) = H^q(C^{p,\bullet}_c(X), \del) = H^{p,q}_{\trop, c}(X).$$
Where we again used Lemma \ref{Cpqsheaf} at both middle equalities. This completes the proof of the proposition. 
\end{proof}

\subsection{Dolbeault cohomology of superforms}\label{sec:Dolbeault}

In this subsection we prove a local exactness result for superforms on polyhedral spaces, called the Poincar\'e lemma. This  extends the Poincar\'e Lemma for superforms on polyhedral complexes in $\mathbb{R}^r$ from  \cite[Theorem 2.16]{Jell}.  Using the acyclicity established in the last section, we identify the Dolbeault cohomology of superforms with the cohomology of certain sheaves. We then calculate the sections of these sheaves on basic open sets in Proposition   \ref{lem:Lp}.

\begin{satz} [Poincar\'e lemma] \label{thm:PLemmaX}
Let $X$ be a polyhedral space and $U \subset X$ an open subset. Let $\alpha \in \AS^{p,q}(U)$ with $q > 0$ and $d''\alpha = 0$. Then for every $x \in U$ there exists an open subset $V \subset X$ with $x \in V$ and a superform $\beta \in \AS^{p,q-1}(V)$ such that $d'' \beta = \alpha|_{V}$.
\begin{proof}
After shrinking $U$, we may assume that there is a chart $\varphi \colon U \rightarrow \Omega$ for $\Omega$ an open subset of the support of a polyhedral complex $\CC$ in $\T^r$. 
Since this question is purely local, we may prove the statement for $x \in \Omega$, where $\Omega$ is an open subset  
of $|\CS|$ for a polyhedral complex $\CC$ in $\T^r$. 

For $\CS$ a polyhedral complex in $\R^r$ the statement of the theorem is proven in  \cite[Theorem 2.16]{Jell}.
By replacing $\CS$ by $\CS_\emptyset$ and $\Omega$ by $\Omega_\emptyset$, we conclude that  if $\sed(x) = \emptyset$ the theorem holds.

For the general case, let $I = \sed(x)$ and after possibly shrinking $\Omega$ we may assume that $I$ is the unique maximal sedentarity among points in $\Omega$
and $\alpha$ is determined by $\alpha_I$ on $\Omega$. After possibly shrinking $\Omega$ again, by the case $I = \emptyset$, there is a superform $\beta_I \in \AS^{p,q}(\Omega_I)$ such that $d'' \beta_I = \alpha_I$. For each $J \subset I$, set $\beta_J = \pi^*_{IJ} \beta_I$.  Then this determines  a superform $\beta \in \AS^{p,q}(\Omega)$ and since the affine pullback commutes with $d''$, we have $d'' \beta_J 
= \alpha_J$, hence $\beta$ has the required property and the theorem is proven. 
\end{proof} 
\end{satz}

\begin{defn} \label{defn:LSp}
For $X$ a polyhedral space and $p \in \mathbb{N}$ we define the sheaf
\begin{align*}
\LS^p_X := \ker (d'' \colon \AS_X^{p,0} \rightarrow \AS_X^{p,1}).
\end{align*}
Again we omit the subscript $X$ on $\LS^p_X$ if the space $X$ is clear from the context. 
\end{defn}

\begin{kor} \label{kor:PLemmaX}
For a polyhedral space $X$ and  all $p \in \mathbb{N}$, the complex 
\begin{align*}
0 \rightarrow \LS^p \rightarrow \AS^{p,0} \overset {d''} {\rightarrow} \AS^{p,1} \overset {d''} {\rightarrow} \AS^{p,2} \rightarrow \dots
\end{align*}
of sheaves on $X$ is exact. Furthermore it is an acyclic resolution, we thus have canonical isomorphisms
$$
H^q(X, \LS^p) \cong H^{p,q}_{d''}(X) \qquad \text{ and} \qquad
H^q_c(X, \LS^p) \cong H^{p,q}_{d'',c}(X).$$
\end{kor}
\begin{proof}
Exactness is a direct consequence of Theorem \ref{thm:PLemmaX} and Definition \ref{defn:LSp}. Acyclicity follows from Proposition \ref{rem:complexonX}. 
\end{proof}

\begin{bsp} \label{ExamplepartIII}
We calculate the 
dimensions 
of the Dolbeault cohomology for the polyhedral spaces from Example \ref{ex:part2}. 
Let $h^{p, q}(X) := \dim{H^{p, q}_{d''}}(X)$ for all $p, q$. 

It is easy to see that, for any polyhedral space,  $\LS^0$ is the constant sheaf with stalk ${\R}$. By Corollary \ref{kor:PLemmaX} and comparison with singular cohomology, we obtain $h^{0,0}(X) = 1$ and $h^{0,1}(X) = 0$. This argument shows that, in general, the cohomology groups  $H^{0,q}(X)$ do not depend on the atlas of $X$. 
For the polyhedral space $X$ from Example  \ref{ex:part2}, recall that the compatibility condition for superforms along the boundary strata implies that all smooth functions are locally constant at points $0, 1 \in X$. Also all superforms of positive degree have support away from the boundary points.
Thus  $(1,0)$ and $(1,1)$-superforms on $X$ are simply forms on $\R$ with compact support. Fix a coordinate $x$ on $\R$. Then $\alpha \in \AS_c^{1,0}(\R)$ is of the form $\alpha = f d'x$ with $f \in C_c^{\infty}(\R)$ and is closed precisely if $\frac {\del f} {\del x} = 0$. This means $f = 0$ and hence $\alpha = 0$, thus $h^{1,0}(X) = 0$. 
For $h^{1,1}(X)$ note that a superform $f d'x \vedge d''x$ with $f \in C_c^{\infty}(\R)$ is exact
precisely if $f$ has an antiderivative with compact support in $\R$.
 This is the case  when $\int _\R f = 0$, thus $h^{1,1}(X) = 1$.  Notice that  $h^{0,0}(X) = h^{1,1}(X)$ and $h^{0,1}(X) = h^{1,0}(X)$.   
Example \ref{ex:part2} 
also considered  the polyhedral space $\tilde{X}$ given by $[0,1]$ with the inclusion $[0,1] \inj \R$ as the only chart. 
The dimensions of the cohomology groups for $\tilde{X}$ are 
$$h^{0,0}(\tilde{X}) =1,  \quad h^{0,1}(\tilde{X}) = 0,  \quad h^{1,0}(\tilde{X}) = 1 \quad \text{and} \quad h^{1,1}(\tilde{X}) = 0.$$
We will revisit this in Example  \ref{ExamplepartIV}.
\end{bsp}

\begin{prop}\label{lem:Lp}
Let $\CS$ be a polyhedral complex in $\T^r$ and $\Omega$ be basic open set of $|\CS|$
 with minimal polyhedron $\sigma$ of sedentarity $I$. Then we have
\begin{align*}
\LS^p (\Omega) = \left( \sum \limits_{\tau \in  \CC_I: \sigma \prec \tau} \bigwedge^p \Linear(\tau) \right)^*.
\end{align*}
For basic open subsets $\Omega'\subset\Omega$,
 the restriction maps $\LS^p (\Omega) \rightarrow \LS^p (\Omega')$ are given by the dual of the inclusion
$$\sum \limits_{\tau \in  \CC_I: \sigma' \prec \tau} \bigwedge^p \Linear(\tau) 
\inj \sum \limits_{\tau \in  \CC_I: \sigma \prec \tau} \bigwedge^p \Linear(\tau)$$
when the minimal polyhedron $\sigma'$ of $\Omega'$ is also of sedentarity $I$. 
If the sedentarity of $\sigma'$ is $J \subsetneq I$ then the restriction map is dual to the map 
$$\sum \limits_{\tau \in  \CC_J: \sigma' \prec \tau} \bigwedge^p \Linear(\tau) 
\rightarrow \sum \limits_{\tau \in  \CC_I: \sigma \prec \tau} \bigwedge^p \Linear(\tau)$$
which is the composition of projection $\pi_{IJ}$ and the above inclusion. 
\begin{proof}
We start with the case $I = \emptyset$, thus $\Omega \subset \R^r$. 
Given  a $(p, 0)$-superform 
in the kernel of $d''$,
 the strategy is to construct a superform whose coefficient functions are all constant and to show that this superform agrees with the original superform on $\Omega$.

Recall that $\sigma$ is the minimal polyhedron of the basic open set $\Omega$. 
Set $V = \sum \limits_{\sigma \prec \tau} \bigwedge^p \Linear(\tau)$. There is a  natural map $V^* \rightarrow \LS^p(\Omega) \subset \AS^{p,0}(\Omega)$ and this is clearly injective. To show surjectivity choose $v_1,\dots,v_k$ such that each $v_i \in \bigwedge^p \Linear(\tau)$ for some $\tau$ and $v_1,\dots,v_k$ is a basis of $V$ and extend to a basis $v_1,\dots,v_k,v_{k+1},\dots,v_s$ of $\bigwedge^p \R^r$. 
Write $\alpha \in \LS^p(\Omega)$ as
\begin{align*}
\alpha = \sum \limits_{i = 1} ^{s} f_i d'v_i
\end{align*}
for
$f_i$ smooth functions on open subsets of $\R^r$. Here $d'v_1,\dots,d'v_k$ is the dual to the fixed basis of $\bigwedge^p \R^r$. By definition we have 
$f_i = \langle \alpha, v_i \rangle$.

 Notice that for any $\sigma \prec \tau$, the set $\Omega \cap \tau$ is connected, since it is the intersection of an open cube and a polyhedron. 
 For a fixed $\tau$ such that $\sigma \prec \tau$ and a fixed vector $w_{\tau} \in  \bigwedge^p \Linear(\tau)$ 
 define the  function 
\begin{align}\label{eqn:functiononfaces}
\langle \alpha, w_\tau \rangle: \tau \cap \Omega \rightarrow \R.
\end{align}
The closedness of $\alpha$ implies that this function is constant over
all $x \in \tau \cap \Omega$.
Fix a point $x \in \sigma \cap \Omega$, then define
  $c_i := f_i(x)$ 
  and  $\alpha' := \sum \limits_{i=1} ^k c_i d'v_i$.

We want to show that $\alpha$ and $\alpha'$ are equivalent when restricted to $\Omega$. Then we are done because $\alpha'$ is certainly in the image of $V^*$. 
For any $\tau \in \CC$ and any $w_{\tau} \in \bigwedge^p \Linear (\tau)$ write 
$w_{\tau} = \sum \limits_{i=1} ^k\lambda_i v_i$. Then for any $y \in \Omega$ such that $y \in \inte(\tau)$ we have 
\begin{align*} 
\langle \alpha, w_\tau \rangle (y) & {=} \langle \alpha,  w_\tau \rangle (x)  = \sum \limits_{i=1} ^k \lambda_i \langle \alpha, v_i \rangle (x)  {=} \sum \limits_{i=1} ^k \lambda_i f_i(x) \\
&= \sum \limits_{i=1} ^k\lambda_i c_i = \sum \limits_{i=1} ^k \lambda_i \langle \alpha', v_i \rangle(y) = \langle \alpha',  w_\tau \rangle (y).
\end{align*}
The first equality follows because the function defined in (\ref{eqn:functiononfaces}) is constant. The second equality follows by the definition of $\lambda_i$. The third equality follows from the fact that
$f_i = \langle \alpha, v_i \rangle$. The final equality also follows by definition. 
Therefore, $\alpha$ and $\alpha'$ are equivalent when restricted to $\Omega$. 

For the general case $I \neq \emptyset$, first we apply the above argument to $\Omega_I$ which is a basic open subset of the polyhedral complex $\CC_I$ by Lemma \ref{prop:basicopen}. Writing $X = |\CC|$ and $X_I = |\CS_I|$  
 we obtain
\begin{align*}
\LS_{X_I}^p (\Omega_I) = \left( \sum \limits_{ \tau \in \CC_I: \sigma \prec \tau} \bigwedge^p \Linear(\tau) \right)^*.
\end{align*}
Thus we only have to show
\begin{align*}
\LS_{X_I}^p(\Omega_I) \cong \LS_X^p(\Omega).
\end{align*}

Using the pullbacks of the projection maps define
\begin{align*}
\LS_{X_I}^p(\Omega_I) &\rightarrow \LS_X^p(\Omega) \\
\alpha_I &\mapsto (\pi_{IJ}^*\alpha_I)_{J \subset I}.
\end{align*}
This is clearly well defined and injective, we thus have to show surjectivity. 
More precisely, for $\alpha \in \LS_X^p(\Omega)$, it remains to show that $\alpha_J|_{\Omega_J \cap \tau} = \pi^*_{IJ}(\alpha_I|_{\Omega_I \cap \tau})$ for all $J \subset I$ and $\tau$ such that $\sigma \prec \tau$. 
By the condition of compatibility for $\alpha$ there exists a neighborhood $\Omega_{x}$ of $x$ such that
\begin{align*}
\alpha_J|_{\Omega_{x,J}} = \pi^*_{IJ}( \alpha_I|_{\Omega_{x,I}}),
\end{align*}
hence in particular
\begin{align*}
\alpha_J|_{\Omega_{x,J} \cap \tau } = \pi^*_{IJ}( \alpha_I|_{\Omega_{x,I} \cap \tau}).
\end{align*}
Since $\Omega_I \cap \tau $ is connected, the restriction $\LS_{X_I}^p(\Omega_I \cap \tau) \rightarrow \LS_{X_I}^p(\Omega_{x,I} \cap \tau )$ is injective and similarly if we replace $I$ with $J$.  Thus we have
\begin{align*}
\alpha_J|_{\Omega_J\cap \tau} = (\pi_{IJ}^* \alpha_I|_{\Omega_I \cap \tau}), 
\end{align*}
proving that
\begin{align*}
\LS^p (\Omega) = \left( \sum \limits_{\tau \in  \CC_I: \sigma \prec \tau} \bigwedge^p \Linear(\tau) \right)^*.
\end{align*}

For $\Omega' \subset \Omega$, the claim concerning the restriction maps is clear if the minimal polyhedra of $\Omega$ and $\Omega'$ are of the same sedentarity. If the minimal polyhedron $\sigma'$ of $\Omega'$ is of sedentarity $J$, then the restriction $\LS_X^p(\Omega) \rightarrow \LS_X^p(\Omega')$ is given by restriction on each stratum. By  identifying $\LS^p_X(\Omega) \cong \LS^p_{X_I}(\Omega_I)$ and $\LS^p_X(\Omega') \cong \LS^p_{X_J}(\Omega'_J)$ we obtain the claimed restriction maps.
\end{proof}
\end{prop}

\subsection{Equivalence of cohomologies} \label{sec:equivofcoh}

We are now ready to prove that tropical cohomology and Dolbeault cohomology of superforms are isomorphic. We will use the results established in the previous two subsections.

\begin{lem}\label{lem:isosheaves}
Let $X$ be a polyhedral space. 
Then there is a canonical isomorphism of sheaves $\mathcal{L}_X^p \cong \mathcal{F}_X^p$. 
\end{lem}

\begin{proof}
Let $( \varphi_i \colon U_i \rightarrow \Omega_i \subset X_i)_{i \in I}$ be an atlas for $X$ and choose a polyhedral structure $\CS_i $ on $X_i$ for all $i$.  
For $\Omega \subset X$ a basic open subset there is an isomorphism $\mathcal{L}^p(\Omega) \to  \mathcal{F}^p(\Omega)$ by Proposition  \ref{lem:Lp} and 
Lemma \ref{lem:Fp}. 
Also for 
a basic open subset $\Omega'$ contained in $\Omega$, 
the restriction maps  form the following commutative diagram:
\begin{align*}
\begin{xy}
\xymatrix{
\LS^p(\Omega) \ar[rr] \ar[d]^{\cong} && \LS^{p}(\Omega') \ar[d]^{\cong} \\
\FS^p(\Omega) \ar[rr] && \FS^p(\Omega) 
}
\end{xy}
\end{align*}

Thus $\mathcal{L}^p$ and $\mathcal{F}^p$ agree on a basis of the topology of $X$, and by \cite[Proposition I-12, i)]{EisenbudHarris} the two  sheaves agree. 
\end{proof}

Now we arrive at Theorem \ref{introthm1} from the introduction.

\begin{satz}\label{thm:equivalenceDeRhamPQ}
Let $X$ be a  polyhedral space 
equipped with a face structure. Then there are canonical isomorphisms
\begin{align*}
H^{p, q}_{\trop}(X) \cong H^{p,q}_{d''}(X) \qquad  \text{ and} \qquad  
H^{p, q}_{\trop,c}(X) \cong H^{p,q}_{d'',c}(X). 
\end{align*}
\end{satz}

\begin{proof}
Corollary \ref{kor:PLemmaX}  relates the Dolbeault cohomology of superforms on $X$ with the cohomology of
the sheaf $\mathcal{L}^p$. Proposition \ref{Prop:tropsheaf} does the same with tropical cohomology and the cohomology of $\FS^p$. Combining this with Lemma \ref{lem:isosheaves} proves the isomorphisms. 
\end{proof}

\begin{bem}\label{rem:Fpcohomology}
Notice that in the absence of a face structure on $X$, the sheaf cohomology of $\mathcal{F}^p$ and the  Dolbeault cohomology of superforms are still isomorphic by applying Corollary \ref{kor:PLemmaX} and Lemma \ref{lem:isosheaves}.

For a polyhedral space $X$ 
 we have $\LS^0_X = \underline{\R} = \FS^0_X$, where $\underline{\R}$ is the constant sheaf with stalks $\R$. Thus we have $H^{0,q}_{d''}(X) \cong H^{0,q}_{\trop}(X) \cong  H^{q}_{\sing}(X)$ by Proposition \ref{Prop:tropsheaf},  Corollary  \ref{kor:PLemmaX} and \cite[Chapter III, Theorem 1.1]{Bredon}. 

The tropical cohomology groups and Dolbeault cohomology groups of superforms for $p>0$ do however depend heavily on the
 equivalence class of  the chosen atlas, and not just on the topological space underlying a polyhedral space (see Example \ref{ExamplepartIII}). 
\end{bem}

\begin{prop} \label{Prop:functorial}
Let $\CS$ and $\DS$ be polyhedral complexes in $\T^r$ and  $\T^s$ respectively and let $\delta \colon |\DS| \rightarrow |\CS|$ be a map  induced by an extended affine map $\delta \colon \T^r \to \T^s$ such that the image of every face of $\DS$ is a face of $\CS$. Let $X \subset |\CS|$ and $Y  \subset |\DS|$ be open subsets such that $\delta(Y) \subset X$. Then 
there are maps $\delta^*_{d''}: H^{p,q}_{d''}(X) \rightarrow H^{p,q}_{d''}(Y)$ and $\delta^*_{\trop} \colon H^{p,q}_{\trop}(X) \rightarrow H^{p,q}_{\trop}(Y)$. 
Moreover 
the following diagram commutes:
\begin{align*}
\begin{xy}
\xymatrix{
H^{p,q}_{d''}(X) \ar[r] \ar[d]_{\delta^*_{d''}} & H^{p,q}_{\trop}(X) \ar[d]^{\delta^*_{\trop}} \ar[l]\\
H^{p,q}_{d''}(Y) \ar[r] & H^{p,q}_{\trop}(Y). \ar[l]
}
\end{xy}
\end{align*}
If $\delta$ is a proper map, then the same holds for cohomology with compact support. 
\end{prop}
\begin{proof}
We have the following commutative diagram:
\begin{align*}
\begin{xy}
\xymatrix{
\delta^{-1}\AS^{p,\bullet}_X   \ar[d]_{\delta^*_{d''}} & \delta^{-1}\LS^p_X \ar[r] \ar[l] \ar[d]_{\delta^*} & \delta^{-1}\FS^p_X \ar[r] \ar[l]  \ar[d]^{\delta^*} & \delta^{-1}\underline{C}^{p,\bullet}_X \ar[d]^{\delta^*_{\underline{\trop}}} \\
\AS^{p,\bullet}_Y  & \LS^p_Y \ar[r] \ar[l] & \FS^p_Y \ar[r] \ar[l] & \underline{C}^{p,\bullet}_Y. 
 }
\end{xy}
\end{align*}
All horizontal maps are quasi-isomorphisms. Taking hypercohomology of the functor of global sections, we get the cohomology of the complexes of global sections on both the far left and the far right as well as  the cohomology of the sheaves in the middle. This shows the commutativity of the left square in the diagram 
\begin{align*}
\begin{xy}
\xymatrix{
H^{p,q}_{d''}(X) \ar[r] \ar[d]_{\delta^*_{d''}} & H^{q}(\underline{C}^{p,\bullet}(X), \del^*) \ar[d]^{\delta^*_{\underline{\trop}}} \ar[l] \ar[r] & H^{p,q}_{\trop}(X) \ar[l] \ar[d]^{\delta^*_{\trop}} \\
H^{p,q}_{d''}(Y) \ar[r]  & H^{q}(\underline{C}^{p,\bullet}(Y), \del^*)  \ar[l] \ar[r] & H^{p,q}_{\trop}(Y). \ar[l]
}
\end{xy}
\end{align*}
The commutativity of the right square 
 follows by definition of the complexes $(\underline{C}^{p,\bullet}(X), \del^*)$ and tropical cohomology. This proves the claim for usual cohomology. 
 
 If $\delta$ is proper then the pullbacks are well defined for sections with compact support and the above arguments can be applied directly. 
 This completes the proof of the proposition.  
\end{proof}

\begin{Rem}
Another technique to prove Theorem \ref{thm:equivalenceDeRhamPQ} could be to use a map similar to the de Rham map, which provides an isomorphism between the de Rham and singular cohomologies in the classical theory. This de Rham map is given explicitly by
\begin{align*}
H^q_{dR}(X) &\rightarrow H^q_{\sing}(X), \\
[\alpha] &\mapsto \left( (\delta: \Delta_q \rightarrow X) \mapsto \int_{\Delta} \delta^*(\alpha) \right).
\end{align*}
There is a similar map from  spaces of superforms to tropical cochains given by contracting a $(p,q)$-superform by the coefficient of a singular tropical cell. This produces a $(0,q)$-superform which can be  integrated over the simplex following Section \ref{sec:integration}.  
Some care needs to be taken to allow only smooth simplicies (as in the classical case) and also to ensure that the integrals are well defined 
when passing to cohomology.  
We will not do this since it is not required for our considerations, but we  point out that  this could approach could be used to identify the wedge product on Dolbeault cohomology of superforms with the  cup product on  tropical  cohomology. 
\end{Rem}

\section{Poincar\'e duality} \label{chapter4}

In this section we prove Poincar\'e duality for a class of polyhedral spaces, known as  tropical manifolds. 
By Poincar\'e  duality we mean an explicit isomorphism 
$\PD \colon H^{p,q}(X) \rightarrow H^{n-p,n-q}_c(X)^*$.  Just as for  standard differential forms, this map is  defined using a pairing given by integration of superforms. In Subsection \ref{sec:integration}, we show  that the pairing given by integration 
 descends to a  pairing on the 
Dolbeault  cohomology  of superforms on tropical spaces. 
 Finally we show Poincar\'e duality 
for tropical manifolds in Subsection \ref{subsec:PD}.

\subsection{Integration of superforms} \label{sec:integration}

Throughout the next sections we consider the standard lattice $\Z^r \subset \R^r$. Notice that in $\T^r$ there is 
an induced lattice in each stratum $\R^r_I$.
 We define integration of superforms on rational polyhedral complexes in $\T^r$ by extending the theory already developed in $\R^r$. 
 Then we use  partitions  of unity to define the integration of  superforms on rational polyhedral spaces.
We prove a version of Stokes' theorem for superforms on tropical spaces
 which ensures that
 this integration descends to Dolbeault  cohomology.

\begin{lem}\label{lem:intStrata}
Let $X$ be a polyhedral subspace of dimension $n$ 
in $\TT^r$.
Suppose that there exists $I \subset [r]$ such that $X_{I}$ is dense in $X$.
If $\alpha \in \AS^{p,q}_c(X)$ is  such that $\max(p,q) = n$, then $\alpha_I \in \AS^{p, q}(X_I)$ has compact support and 
$\alpha_J = 0$ for each $J \supsetneq I$. 
\end{lem}
\begin{proof}
For $J \neq I$ we have $\dim X_J < n$, 
so for  dimension reasons 
 $\supp(\alpha_J) = \emptyset$. 
Then Lemma \ref{lem:supportboundary} shows $\supp(\alpha) = \supp(\alpha_I)$. 
\end{proof}

\begin{defn}\label{def:rational}
A polyhedral complex $\CC$ in $\T^r$ is called 
\emph{rational}  if every polyhedron $\sigma$ is parallel to a  subspace of  $\R^r_{\text{sed}(\sigma)}$ defined over $\Z$. 

A polyhedral space $X$ with atlas $(\varphi_i \colon U_i \to X_i \subset \T^{r_i})_{i \in I}$ is called \emph{rational} if 
every $X_i$ is the  support of a  rational polyhedral complex and the transition functions $\varphi_i \circ \varphi^{-1}_j$ are integral extended affine maps.
\end{defn}

By definition, for any polyhedron $\sigma$ in a rational polyhedral complex there is a  canonical lattice of full rank  $\Z(\sigma) \subset \Linear(\sigma)$.

\begin{defn}
Let $\CS$ be a polyhedral complex  of pure  dimension $n$. 
We write $\CC_n$ for the set of $n$-dimensional polyhedra.
Then $\CS$ is  \emph{weighted} 
if it is  
equipped with a weight function $w \colon \CS_n \to \Z$.  

A polyhedral space $X$ is \emph{weighted} 
if it is 
equipped with a continuous weight function $w \colon X^\circ \to \Z$, 
which is defined on a dense open subset $X^\circ$ of $X$. Furthermore,  we require that for every chart
$\varphi_i \colon U_i \to X_i \subset \T^{r_i}$ of $X$ and every connected component $U$ of $X^\circ$ the space $\varphi_i(U \cap U_i)$ is the  intersection of 
$\varphi_i(U_i)$ with the relative interior of a polyhedron in $\T^{r_i}$. 

For another open dense subset $Y^{\circ} \subset X$  and weight function 
$w' \colon Y^\circ \to \Z$ we say $w$  
is equivalent to $w^{\prime}$ if $w^{\prime} \vert_{Y^{\circ} \cap X^\circ} = w \vert_{Y^\circ \cap X^{\circ}}$ and
$Y^{\circ} \cap X^{\circ}$ is dense in $X$. 
\end{defn}

Note that  the continuity of  a weight function $w$ ensures that it is  constant on any connected component of $X^\circ$.

\begin{defn}
Let $X$ be a rational weighted polyhedral space with weight function $w\colon X^\circ \to \Z$ and 
charts $(\varphi_i \colon U_i \to |\CS_i|)_{i \in I}$, where the $\CS_i$ are weighted rational polyhedral complex. 
We say that $\{\CS_i \,|\, i \in I \}$ is a \emph{weighted rational polyhedral structure} on $X$ if 
for every connected component $U$ of $X^\circ$ the image $\varphi_i (U \cap U_i) = \varphi_i(U_i) \cap \inte(\sigma)$ 
for a polyhedron $\sigma$ in $\CS_i$.  
Moreover, we require that  
$w|_U  \equiv m_{\sigma}$, where  $m_{\sigma}$ is the weight of $\sigma$ in $\CS_i$. 
\end{defn}

For a concrete choice of a weighted rational polyhedral structure on $X$, the weights of faces of $\CS_i$ outside of $\varphi_i(U_i)$ will not matter  for any of our constructions.

Now we recall the definition of integration of superforms on weighted polyhedral complexes in $\R^r$ from Chambert-Loir and Ducros \cite{CLD} but follow the notation of Gubler \cite{Gubler}.  We also extend the definition of integration  to polyhedral complexes in $\T^r$.

\begin{defn} \label{def:integral}
 Let $\CS$ be a pure $n$-dimensional weighted rational polyhedral complex in $\R^r$. 
\begin{enumerate}
\item
Let $\alpha \in  \AS^{n,n}_c(|\CC|)$.  For $\sigma \in \CC_n$, choose a basis $x_1,\dots,x_n$ of $\Z(\sigma)$. Then $\alpha|_\sigma$ can be written as 
\begin{align*}
f_\alpha d'x_1 \vedge d''x_1  \vedge \dots \vedge d''x_n = (-1)^{\frac{n(n-1)}{2}} f_\alpha d'x_1 \vedge d'x_2 \vedge \dots \vedge d''x_{n-1} \vedge d''x_n
\end{align*}
for $f_\alpha \in \AS^{0,0}_c(\sigma)$. Since this is an integral basis, $f_\alpha$ is independent of the choice of $x_1,\dots,x_n$. Then the \emph{integral of $\alpha$ over $\sigma$} is
\begin{align*} 
\int \limits_\sigma \alpha := \int \limits_\sigma f_\alpha,
\end{align*}
where the integral on the right is taken with respect to the volume defined by the lattice $\Z(\sigma) \subset \Linear(\sigma)$. 
The integral over the weighted rational polyhedral complex $\CS$ is
$$ \int \limits_{\CS} \alpha := \sum \limits_{\sigma \in \CS_n} m_\sigma \int \limits_\sigma \alpha,$$
where $m_\sigma$ is the weight of $\sigma$. 
\item
Let $\tau \prec \sigma$ be a face of $\sigma$ of codimension one. Denote by $\nu_{\tau, \sigma} \in \Z(\sigma)$ a representative of the unique generator of $\Z(\sigma) / \Z(\tau)$ which points inside of $\sigma$. Then for $\beta \in \AS^{n,n-1}_c(|\CC|)$  the  \emph{boundary integral of $\beta$ over $\partial \sigma$} is
\begin{align*}
\int \limits_{\del \sigma} \beta = \sum \limits_{\tau \prec \sigma} \int \limits_\tau \langle \beta; \nu_{\tau, \sigma} \rangle_{1},
\end{align*}
where on the right hand side we use the integral of the $(n-1,n-1)$-form $\langle \beta; \nu_{\tau, \sigma} \rangle_{n}$ over the $(n-1)$-dimensional polyhedron $\tau$ as defined in $i)$. The integral over the boundary of the weighted rational polyhedral complex  $\CS$ is $$ 
\int \limits_{\del \CS} \beta := \sum \limits_{\sigma \in \CS_n} m_\sigma \int \limits_{\del \sigma} \beta,$$
where $m_\sigma$ is the weight of $\sigma$. 
\item
If $\CS$ is a weighted rational polyhedral complex in $\TT^r$, then the definitions from $i)$ and $ii)$ can be extended. Note  that if
$\alpha \in \AS^{n,n}_c(|\CC|)$ 
 and $\sigma \in \CS_n$ is the closure of $\sigma' \in \CS_{I,n}$, then the support of $\alpha|_{\sigma}$ is contained in $\sigma'$ by Lemma \ref{lem:intStrata} and we define 
\begin{align*}
\int \limits_\sigma \alpha := \int \limits_{\sigma'} \alpha|_{\sigma'}.
\end{align*}
The same works for integrals of 
$(n, n-1)$-forms with compact support 
 over codimension one  faces. Note that if the codimension one  face $\tau$ of $\sigma$ is of a higher sedentarity than $\sigma$, then by  Lemma \ref{lem:intStrata}, for $\beta \in \AS^{n, n-1}_c(|\CC|)$ the restriction $\beta|_\sigma$ has support away from $\tau$.
Thus for $\int_{ \del \sigma} \beta$ 
we only integrate over codimension one faces of $\sigma$ which have the same sedentarity as $\sigma$. 

\end{enumerate}
\end{defn}

Now integration on polyhedral spaces can be defined using the integration on polyhedral subspaces and partitions of unity just as in manifold theory. 

\begin{defn}
\label{def:integration}
Let $X$ be a pure $n$-dimensional  weighted rational polyhedral space with weight function $w$ and  atlas $(\varphi_i : U_i \rightarrow |\CS_i|)_{i \in I}$ where $\{\CS_i\}_{i \in I}$ is a weighted rational polyhedral structure on $X$. 
Let $\alpha \in \AS_c^{n,n}(X)$ and $(f_j)_{j \in J}$ be a partition of unity with functions in $\AS^{0,0}$ subordinate to the cover $(U_i)$ as in Lemma \ref{lem:partitionsofunity}. 
Define the \emph{integral of $\alpha$ over $X$} by 
\begin{align*}
\int \limits_{X} \alpha := \sum \limits_{j \in J} \; \; \int \limits_{\CS_i} \alpha_j,
\end{align*}
where $\alpha_j \in \AS_c^{n,n}(|\CS_i|)$  is the superform corresponding to $f_j\alpha \in \AS_c^{n,n}(U_i)$. 
Since $\alpha$ has compact support the sum on the right hand side is finite. 
The integral on the right is defined in Definition \ref{def:integral}.  
\end{defn}

Notice that the above definition is dependent on the choice of weight function $w$. 
However, the following lemma 
ensures the above defined integral is independent of the choice of charts and partition of unity
on the polyhedral space.

\begin{lem}
Let $X$ be a pure $n$-dimensional weighted rational  polyhedral space. 
Then the integral from Definition \ref{def:integration} is independent of the choice of atlas, weighted rational polyhedral structure, and partition of unity.
\end{lem}
\begin{proof}

Consider an atlas $(\varphi_i \colon U_i \to |\CS_i|)_{i \in I}$, where $\{ \CS_i \}$ is a weighted rational polyhedral structure on $X$.
For $i, j \in $ let $F:= \varphi_j \circ \varphi_i^{-1}$, $U := U_i \cap U_j$ and $\alpha \in \AS^{p,q}_c(U)$.
Denote by $\alpha_i$ and $\alpha_j$ the forms corresponding to $\alpha$ on $\varphi_i(U)$ and $\varphi_j(U)$ respectively.
We claim that $\int_{\CS_i} \alpha_i = \int_{\CS_j} \alpha_j$. 

We may now assume that $\CS_i$ is a polyhedral complex in $\R^r$. This is because the pullback of a superform $
\alpha$ on a polyhedral complex in $\T^r$ is defined by pulling back the components  $\alpha_I$ along affine 
maps and integration is also defined by considering the intersection of the polyhedral subspace with the vector 
spaces $\R^r_I$.

Let $F_* \CC_i$ be the pushforward of $\CC_i$ in the sense of weighted polyhedral complexes \cite[3.9]{Gubler}.
Since $F$ is an isomorphism between $\varphi_i(U)$ and $\varphi_j(U)$, the weights are constructed from the weight function of $X$,
and $\supp(\alpha_j)$ is contained in $\varphi_j(U)$, 
we have $\int_{\CS_j} \alpha_j = \int_{F_* \CS_i} \alpha_j $. 
By the projection forumla \cite[Proposition 3.10]{Gubler} we have $\int_{F_* \CS_i} \alpha_j = \int_{\CS_i} F^*\alpha_j$. 
By definition we have $\alpha_i = F^* \alpha_j$.
This shows $\int_{\CS_i} \alpha_i = \int_{\CS_j} \alpha_j$.

The lemma now follows by the standard argument for classical manifolds.
Given two atlases, two weighted rational polyhedral structures and two partitions of unity subordinate to their respective covers we can consider the union of the atlases and weighted rational polyhedral structures. We can also form a partition of unity for this new atlas by multiplying the given partitions of unity. 
The argument above then shows independence of the integral.
\end{proof}

\begin{defn}\label{def:tropicalVariety}
Let $\CS$ be a weighted rational polyhedral complex in $\TT^r$ which is pure $n$-dimensional. Let $\tau$ be a face of $\CS$ of dimension $n-1$. 
We say that \textit{$\CS$ is balanced at $\tau$} if 
\begin{align*}
\sum \limits_{\sigma \in \CS_n: \tau \prec \sigma;  \sed(\sigma) = \sed(\tau)
} m_\sigma \nu_{\sigma, \tau} \in \Z(\tau),
\end{align*}
with $\nu_{\sigma,\tau}$ as introduced in Definition \ref{def:integral} $ii)$. 
Note that the above sum is well defined since we sum only over faces $\sigma$ having the same sedentarity as $\tau$. 
We say that $\CS$ satisfies the \textit{balancing condition} if it is balanced at every face of dimension $n-1$.

Let $X$ be a weighted rational polyhedral space with atlas $(\varphi_i : U_i \rightarrow |\CS_i|)_{i \in I}$ 
where $\{\CS_i\}_{i \in I}$ is a weighted rational polyhedral structure on $X$. 
Then $X$ is a \emph{tropical space} if for all $i$, the weighted polyhedral complex $\CS_i$ is balanced
at each face which intersects $\varphi_i(U_i)$. 

A tropical space equipped with a single chart will also be called a \emph{tropical cycle}. 
\end{defn}

It follows from the next theorem that whether or not a weighted polyhedral space is a tropical space does not depend on the choice of a weighted rational polyhedral structure on $X$.

\begin{satz} [Stokes' theorem for tropical spaces] \label{thm:StokesTropical}
Let $X$ be an $n$-dimensional weighted rational polyhedral space. Then $X$ is a tropical space 
if and only if for all  $\beta \in \AS_c^{n,n-1}(X)$ we have 
\begin{align*}
\int \limits_{X} d'' \beta = 0.
\end{align*}
\end{satz}

\begin{proof}
The analogous statement for a weighted rational  polyhedral complex in $\R^r$ is true
 \cite[Proposition 3.8]{Gubler}.
If $\CS$ is a polyhedral complex in $\T^r$, 
denote by $I_1,\dots,I_k$ the subsets of $[r]$ such that there exist maximal faces of sedentarity $I_j$. 
Then $\CS$ is balanced if and only if each $\CS_{I_j}$ is balanced and integration over $\CS$ is just the sum of 
the integration over the $\CS_{I_j}$ (for details see \cite[Lemma 1.2.30 \& Remark 2.1.48]{JellThesis}). 
By considering each $\CS_{I_j}$ separately, we may assume that 
each facet of $\CS$ is of the same sedentarity. Without loss of generality, we may also suppose that this sedentarity is the empty set. 
Then the  balancing condition is a condition on codimension one faces of $\CS_{\emptyset} \subset \R^r$.
Once again by   \cite[Proposition 3.8]{Gubler} the balancing condition is  equivalent to the vanishing of the 
integral 
$\int_{\CS_\emptyset} d'' \beta $ for all $\beta \in \AS_c^{n,n-1}(\vert \CS_{\emptyset} \vert)$. 
Given  $\beta \in   \AS_c^{n,n-1}(\vert \CS \vert)$,  the form  $\beta_{\emptyset}$ also has compact support by  
Lemma \ref{lem:intStrata}.
The equality of the integrals 
$$\int \limits_{\CS} d'' \beta =   \int \limits_{\CS_{\emptyset}} d'' \beta_{\emptyset} $$
follows from part $iii)$ of Definition \ref{def:integral}. 

The statement  for polyhedral spaces follows from the  linearity of both $d''$ and integration.
\end{proof}

\begin{bem}\label{bem:defPD}
Let $X$ be a tropical space of dimension $n$. There is a product 
\begin{align*}
\AS^{p,q}(X) \times \AS^{n-p,n-q}_c(X) &\rightarrow \R, \\
(\alpha, \beta) &\mapsto \int \limits_X \alpha \vedge \beta.
\end{align*}
By Stokes'  theorem \ref{thm:StokesTropical}, given  $\alpha \in \AS^{p,q}(X)$ and $\beta \in \AS^{n-p,n-q-1}_c(X)$ we have 
\begin{align*}
0 = \int \limits_X d''(\alpha \vedge \beta) = \int \limits_X d'' \alpha \vedge \beta + \int \limits_X (-1)^{p+q} \alpha \vedge d''\beta, 
\end{align*}
so that 
\begin{align*}
\int \limits_X d'' \alpha \vedge \beta = (-1)^{p+q+1} \int \limits_X  \alpha \vedge d'' \beta.
\end{align*}
\end{bem}

\begin{defn}\label{def:PoincareMap}
Let $X$ be a tropical space of dimension $n$.
We define 
\begin{align*}
\PD \colon \AS^{p,q}(X) &\rightarrow \AS^{n-p,n-q}_c(X)^*,  \\
\alpha &\mapsto \left( \beta \mapsto  \varepsilon \int \limits_X \alpha \vedge \beta \right)
\end{align*}
where $\AS^{n-p,n-q}_c(X)^* := \Hom_\R(\AS^{n-p,n-q}_c(X), \R)$ denotes the (non-topological) dual vector space of 
$\AS^{n-p,n-q}_c(X)$ and $\varepsilon = (-1)^{p+q/2}$ if $q$ is even and $\varepsilon = (-1)^{(q+1)/2}$ if $q$ is odd. 
 Our choice of $\varepsilon$ together with the Leibniz rule and Stokes' theorem implies that we have a morphism of complexes
\begin{align*}
\PD \colon \AS^{p,\bullet}(X) \rightarrow \AS^{n-p, n-\bullet}_c(X)^*,
\end{align*}
where the dual complex is equipped with the dual differential. 
We now get a map in cohomology
\begin{align*}
\PD \colon H^{p,q}_{d''}(X) \rightarrow H^{n-p, n-q}_{d'',c}(X)^*,
\end{align*}
since we have
\begin{align*}
H^q({\AS^{n-p,n-\bullet}_{c}}(X)^*, {d''}^*) = (H_q(\AS^{n-p,n-\bullet}_c(X), d''))^* = H^{n-p,n-q}_{d'',c}(X)^*.
\end{align*}
\end{defn} 

We will show that the map $\PD$ is an isomorphism for tropical manifolds in the next section.

\subsection{Poincar\'e duality for tropical manifolds} \label{subsec:PD}

In this section we will let   $H^{p,q}(X)$ denote $H^{q}(X, \LS^p)$. By Corollary  \ref{kor:PLemmaX},  we have $H^{p, q}(X) =  H^{p,q}_{d''}(X)$. 
If  $X$ has a face structure, the 
tropical cohomology groups of $X$ are defined and 
  Theorem  \ref{thm:equivalenceDeRhamPQ}  provides canonical isomorphisms  $$H^{p,q}(X) \cong H^{q}(X, \mathcal{L}^p) \cong H^{q}(X, \FS^p) \cong
 H^{p,q}_{\trop}(X).$$
 Similarly, in the setting of cohomology with compact support, let $H^{p,q}_c(X)$ denote $H^{q}_c(X, \LS^p)$.

 We will show that the Poincar\'e duality map given in Definition \ref{def:PoincareMap} is an isomorphism for tropical manifolds.

\begin{defn}
An $n$-dimensional tropical  space $X$   \textit{has Poincar\'e duality} ($\PD$) if
the Poincar\'e duality map 
\begin{align*}
\PD \colon H^{p,q}(X) \rightarrow H^{n-p,n-q}_c(X)^* 
\end{align*}
is an isomorphism for all $p,q$. 
\end{defn}

\begin{bsp} \label{ExamplepartIV} 
Consider again the polyhedral spaces from Example \ref{ExamplepartIII}.  In these examples, the underlying topological space is $[0,1]$, hence compact. Therefore, cohomology and cohomology with compact support are isomorphic.

First take the  charts for $X$ in such a way that $[0,1]$ is the gluing of two copies of $\T$ as in Example \ref{ex:part2}. Letting $h_{\bullet}^{p,q}$ denote $\dim H^{p,q}_{\bullet}(X)$,  then by Example  \ref{ExamplepartIII} we have the equalities $h^{0,0}(X) = 1 = h^{1,1}(X)$ and $h^{1,0}(X) = 0 =  h^{0,1}(X)$.  
If we equip $X$ with the weight function equal to one everywhere, then $X$ is a tropical space.   
It is easy to see that the integration pairing $H^{0,0}(X) \times H^{1,1}(X) \rightarrow \R$ is non-degenerate and so $X$ has $\PD$.  

Alternatively, we can consider the weighted polyhedral space $X$ defined by taking a single chart on $[0,1]$ given by the inclusion of the interval into $\R$ and again equip $X$ with  weight one. However, this does not yield a tropical space since it 
does not satisfy Stokes' theorem. 
 Thus the $\PD$ map is not defined on cohomology. We already saw in Example \ref{ExamplepartIII} that the dimensions of the respective cohomology groups do not agree. 

There are  examples of tropical spaces 
which do not satisfy $\PD$. 
Take for example $Y$ to be the union of the coordinate axes in $\R^2$, again with weight one on each facet. Then it is clear that  $h^{0,0}(Y) = 1$, since $H^{0, 0}(Y)$ is the usual cohomology group $H^0_{\sing}(Y;\R)$. However, it can be shown that  $h^{1,1}_c(Y) = 2$.
\end{bsp}

The rest of this section is devoted to proving Theorem \ref{introthm:Poincare}, which states that tropical manifolds have Poincar\'e duality.
Tropical manifolds are tropical spaces locally modeled on matroidal fans
(see Definition \ref{def:tropicalmanifold}). 

Matroids are a combinatorial abstraction of the notion of independence in mathematics \cite{Oxley}   and every matroid has a representation as a fan tropical variety 
 \cite{Sturmfels:Poly}. Given a matroid $M$ there are explicit constructions of different polyhedral structures for this fan coming from the matroid \cite{FeichtnerSturmfels, ArdilaKlivans}.  In what follows, matroidal fans are always considered to be weighted polyhedral complexes, whose weights are equal to  one on all facets.

\begin{defn} 
A  tropical 
cycle 
in $\R^r$ is \emph{matroidal} 
if it is the support of a matroidal fan $\Sigma$ in $\R^r$ and its weight function is equal to one. 
\end{defn}

\begin{defn} \label{def:tropicalmanifold}
A \textit{tropical manifold} is a tropical space $X$ of dimension $n$ whose weight function is equal to one and has an atlas $(\varphi_i \colon U_i \rightarrow \Omega_i \subset X_i)_{i \in I}$
such that  for all $i$ the spaces  $X_i = \T^{r_i} \times V_i$ where $V_i$ are matroidal 
tropical cycles 
of dimension $n- r_i$ in $\R^{s_i}$. 
\end{defn}

\begin{bsp}

Tropical projective space $\T P^r$ is a tropical manifold using the atlas constructed in Example \ref{ex:tropicalprojectivespace}.

Consider the tropical line $L \subset \R^2$ from Example \ref{ex:tropicalline1} and equip each edge with weight equal to one. The resulting weighted polyhedral 
complex defines
  a tropical manifold since it is the  support of  the matroidal fan associated with the uniform 
matroid $U_{2,3}$ of rank $2$ on $3$ elements. We can also consider the closure of the tropical line in $\mathbb{R}^2$ in tropical projective space $\T P^2$. 
The result is again a tropical manifold, with charts given by restrictions of the charts for $\T P^2$. 

\end{bsp}

We begin by showing that matroidal
 cycles in $\R^r$ have Poincar\'e duality.  
To do this we use an alternative recursive description
 of  matroidal 
cycles
via an operation known as tropical modification  \cite{BIMS15}.  
In the language of matroids, this operation is   related to deletions and contractions. 

\begin{const}
We now describe tropical modifications. Let $W \subset \R^{r-1}$ be a tropical cycle 
and $P\colon \R^{r-1} \to \R$ a piecewise integer affine function. The graph $\Gamma_P(W) \subset \R^{r}$ is the support of a weighted rational polyhedral complex. The weight function on  $\Gamma_P(W)$ is inherited from the weights of $W$. In general, this graph does not satisfy the balancing condition because $P$ is only a piecewise affine function. However, the graph $\Gamma_P(W)$ can be completed to a tropical cycle  
 $V$ in a canonical way. At a codimension one face $E$ of $\Gamma_P(W)$ that does not satisfy the balancing condition, we can attach a facet to $E$ generated by the direction $-e_r$. This facet can be equipped with a unique integer weight so that the resulting polyhedral complex is now balanced at $E$. Applying this procedure at all codimension one faces of $\Gamma_P(W)$ produces a tropical cycle  
 $V$. Notice that there is a map $\delta\colon V \to W$ induced by the linear projection. 
\end{const}

\begin{defn}\label{def:modificationsMat}
Let $W \subset \R^{r-1}$ be a tropical space and $P \colon \R^{r-1} \to \R$ a piecewise integer affine function, then the \textit{open tropical modification} of $W$ along $P$ is the map $\delta\colon V \to W$ where $V\subset \R^r$ is the tropical cycle
 described above. 
The \emph{divisor} of a piecewise integer affine function $P$ restricted to $W$ is the tropical 
space $\text{div}_W(P) \subset W$  supported on the points $w \in W$ such that $ \delta^{-1}(w)$ is a half-line. The  
weight function on 
 $\text{div}_W(P)$ is inherited 
from the tropical cycle 
$V$.  
We also say that $\text{div}_W(P)$ is the \emph{divisor of the modification $\delta$}.

A
\emph{closed tropical modification} is a map $\overline{\delta} \colon \overline{V} \to W$ where $\overline{V} \subset \R^{r-1} \times \T$ is the closure of $V$ and $\overline{\delta}$ is the extension of an open tropical modification $\delta \colon V \to W$. 

A \emph{matroidal tropical modification} is a  modification  where  $V, W,$ and $\text{div}_W(P)$ are all matroidal. 
\end{defn}

\begin{bem}\label{rem:matroidmod}
Suppose the underlying matroid of $V$ is $M$. 
Given a matroidal tropical modification, the  matroid of $W$ is the deletion matroid $M \backslash e$ for some element $e$.  
Moreover, the matroid of  $\text{div}_W(P)$ is the contraction $M |  e$.

Note that for a closed tropical modification $\overline{\delta}: \overline{V} \rightarrow W$ with divisor $D$, the map  $\overline{\delta}|_{\overline{V}_r}\colon \overline{V}_r \rightarrow D$ identifies the subspace $\overline{V}_r = \{ x \in V \ |\  x_r = - \infty \}$ with $D$. We thus may also consider $D$ as a subspace of $\overline{V}$. 
\end{bem}

\begin{bsp}
Let $W = \R$ and equip this space with weight function equal to one. Consider the piecewise integer affine function $P\colon \R \to \R$ defined by $P(x) =  \max(0, x)$. 
The graph $\Gamma_P(W) \subset \R^2$ consists of two half lines meeting at the origin in directions $(0,-1)$ and $(1,1)$. 
The weight on each of the half lines when inherited from $W$ is one.  
To balance the graph, we must attach   to $\Gamma_P(X)$ the half line in direction  $(0,-1)$ at the origin in $\R^2$ and equip this half line with weight one. 
The resulting  space is the tropical line $L$ from Example \ref{ex:tropicalline1}. The 
open tropical modification $\delta \colon L \to W$ of $W$ along $P$ is the map induced by the linear projection with kernel generated by $e_2$. 
The divisor of the modification is the origin in $W = \R$ equipped with weight one. 
When equipped with weight function equal to one, the spaces  $\R$, $L$, and the origin  are all matroidal. Therefore $\delta \colon L \to W$ is a matroidal tropical modification. 
\end{bsp}

It follows from the next proposition that  for any matroidal cycle  $V \subset \R^r$ of dimension $n$ there is a sequence of open matroidal tropical modifications $V \to W_1 \to \dots \to W_{r-n} = \R^n$.

\begin{Prop}\cite[Proposition 2.25]{Shaw:IntMat}
\label{prop:matroidalProj}
Let $V \subsetneq \R^r$ be a matroidal cycle, then there is a coordinate direction $e_i$ such that the
linear projection 
$\delta \colon \R^r \to \R^{r-1}$ with kernel generated by $e_i$ is a matroidal tropical
modification $\delta \colon V \to W$ along a piecewise integer affine function $P$, i.e.~$W \subset \R^{r-1}$ and $D = \text{div}_W(P) \subset \R^{r-1}$ are matroidal cycles.
\end{Prop}

	Tropical cohomology is invariant under closed tropical modifications \cite[Theorem 4.13]{Shaw:Surfaces}. The next lemma checks that this isomorphism also applies to cohomology with compact support and that it is compatible with the $\PD$ map.
	
\begin{Prop} \label{prop:closedmodification}
Let $\delta: \overline{V} \rightarrow W$ be a closed matroidal tropical modification for $W \subset \R^{r-1}$ and $\overline{V} \subset \R^{r-1} \times \T$. 
Then there are
isomorphisms 
$$
\delta^* \colon H^{p,q}(W) \rightarrow H^{p,q}(\overline{V}) \qquad  \text{ and} \qquad  
\delta^* \colon  H^{p,q}_c(W) \rightarrow H^{p,q}_c(\overline{V})
$$
which are induced by the pullback of superforms and are compatible with the Poincar\'e duality map.
\end{Prop}

\begin{proof}
The fact that $\delta^*$ is an isomorphism for tropical cohomology is shown in \cite[Theorem 4.13]{Shaw:Surfaces} and \cite[Proposition 5.6]{JRS}. By Proposition 3.24 this also applies to $H^{p,q}$. For the isomorphism of the compactly supported cohomology groups we can apply the arguments used in  \cite[Proposition 5.6]{JRS} for Borel-Moore homology together with linear duality.

To show that the isomorphism  $\delta^*$ is compatible with the Poincar\'e duality map, if suffices to show  that for $\omega \in \AS^{n,n}_c(W)$ we have 
\begin{align*}
\int_{W} \omega = \int_{\overline{V}} \delta^*(\omega). 
\end{align*}
Showing this is sufficient since 
the wedge product
 is compatible with the pullback. We can choose polyhedral structures on $W$ and $\overline{V} \cap \R^r$ so that 
$W$ is the pushforward of $\overline{V} \cap \R^r$ along $\delta$ in the sense of polyhedral complexes. Then the result follows from the projection formula \cite[Proposition 3.10]{Gubler} since 
 the support of $\delta^{-1}(\omega)$ is contained in $\overline{V} \cap \R^r$ by Lemma \ref{lem:intStrata}. This completes the proof. 
\end{proof}

The next statement relates the cohomology with compact support of the  matroidal cycles of an open tropical modification by an exact sequence.

\begin{Prop} \label{Prop:SeqComCoh}
Let $\Omega$ be an open subset of a polyhedral subspace in $\TT^r$. 
For  $i \in [r]$,  let  $D = \Omega \cap \T^r_i$ and $U:= \Omega \setminus D$. Then there exists a long exact sequence in cohomology with compact support
\begin{align*}
\ldots \rightarrow H^{p,q-1}_c(D) \rightarrow H^{p,q}_c(U) \rightarrow H^{p,q}_c(\Omega) \rightarrow H^{p,q}_c(D) \rightarrow H^{p,q+1}_c(U) \rightarrow \ldots 
\end{align*}
\end{Prop}
\begin{proof}
Note first that $U$ and $\Omega$ are polyhedral spaces via the inclusion into $\T^r$. Furthermore, $D$ is a 
 polyhedral space via the inclusion into $\T^r_i$. 
We claim that the natural sequence of complexes
\begin{align*}
0 \rightarrow \AS^{p,\bullet}_{\Omega,c}(U) \rightarrow \AS^{p,\bullet}_{\Omega,c}(\Omega) \rightarrow \AS^{p,\bullet}_{D,c}(D) \rightarrow 0
\end{align*}
is exact. By the condition of compatibility  for superforms along strata, if a superform restricts to $0$ on $D \subset \T^r_i$, then it must be $0$ on a neighborhood of $D$ in $\T^r$. This shows exactness in the middle of the short exact sequence. Both surjectivity of the last map and injectivity of the first map are clear. The statement of the proposition follows by passing to the  long exact cohomology sequence. 
\end{proof}

We can now prove a vanishing lemma for the cohomology of compact support of a matroidal cycle.

\begin{Lem}\label{lem:vanishing}
Let $V \subset \R^r$ be a matroidal cycle of dimension $n$. 
Then for all $p$ 
\[H^{p, q}_c(V) = 0 \text{ if } q \neq n.\]
\end{Lem}
\begin{proof}
The lemma is proven by  induction on $r$, which is the dimension of the ambient space. When $r = 0$  the space $V$ is a point and the assertion holds. 
We now proceed by induction. 
 If $r = n$,  then $V = \R^n$ and we have $H_c^{p,q}(\R^n) = \bigwedge^p {\R^n}^* \otimes H_c^q(\R^n)$, where $H^q_c(\R^n)$ denotes the usual de Rham cohomology with compact support of $\R^n$. 
We have $H_c^q(\R^n) = 0$ unless $q = n$, so the statement holds in this case.  

Otherwise $ r> n$ and we can apply Proposition \ref{prop:matroidalProj} to obtain a tropical modification $\delta: V \rightarrow W$ with  divisor $D \subset W$ such that $D$ and $W$ are  matriodal cycles in $\R^{r-1}$. 
By the induction assumption, $H^{p,q}_c(D) = 0$ unless $q = n-1$ and $H^{p,q}(W) = 0$ unless $q = n$. 
Applying the long exact sequence from Proposition \ref{Prop:SeqComCoh} and replacing $H^{p,q}_c(\overline{V})$ with $H^{p,q}(W)$ we obtain the sequence  
\begin{align*}
\ldots \to H^{p,q-1}_c(D) \rightarrow H^{p,q}_c(V) \rightarrow H_c^{p,q}(W) \to \ldots
\end{align*}
by Proposition \ref{prop:closedmodification}. 
By the vanishing of cohomologies for $W$ and $D$ 
 we obtain that 
$H^{p, q}_c(V) = 0 \text{ if } q \neq n$. This proves the lemma. 
\end{proof}

The next lemma gives a short exact sequence for the $(p,0)$-cohomology groups of matroidal cycles related by an open tropical modification. 
The statement of the lemma uses the contraction of superforms from Definition \ref{defn:contraction}.

\begin{lem}\label{lem:shortexactMat}
Let  $\delta: V \rightarrow W$  be an open tropical modification along a matroidal fan divisor $D \subset W$. Suppose also that $\delta$ is the restriction of  a linear map on $\R^{r} \to \R^{r-1}$ whose kernel is generated by $e_i$, 
then 
\begin{align*}
\begin{xy}
\xymatrix{
0 \ar[r] & H^{p,0}(W) \ar[r]  & H^{p,0}(V) \ar[r]^{\langle \cdot \ ;e_i\rangle_p} & H^{p-1,0}(D)  \ar[r] &0 
}
\end{xy}
\end{align*}
is an exact sequence. 
\end{lem}

\begin{proof}
For a matroidal fan $V$ let ${\bf F}^p(0)$ denote the value of the cellular sheaf ${\bf F}^p$ from Definition \ref{def:multitangentspacePoint} on $V$ at the origin of the fan.
Then $H^{p, 0}(V)  = {\bf F}^p(0)$ and this group is isomorphic to the 
 $p$-th graded piece of the Orlik-Solomon algebra of the associated matroid 
\cite[Lemma 2.2.7]{Shaw:Thesis}, \cite[Theorem 4]{Zharkov:Orlik}. 
The  pieces of the Orlik-Solomon algebras of the deletion and contraction of a matroid satisfy a short exact sequence similar to the one above \cite[Theorem 3.65]{OrlikTerao}. 
The first part of Remark \ref{rem:matroidmod} explains that $W$ is the matroidal cycle of a deletion of the matroid for $V$ and $D$ is the matroidal cycle of the contracted matroid. 
A direct comparison with this short exact sequence shows that  the last map of the  sequence  above is given by 
 the contraction by $e_i$ in the $p$-{th} component (see the proof of \cite[Lemma 2.2.7]{Shaw:Thesis}). This proves the statement of the lemma. 
\end{proof}

The contraction map in the above proof  is not induced by a map on the level of forms. However, for a closed $(p,0)$-form $\alpha$, the form $\langle \alpha; e_i \rangle_p \in \LS^{p-1,0}(V)$ is  the restriction of a unique form  $\beta \in\LS^{p-1,0}(\overline{V})$. We can restrict $\beta$ to $D$, since we can identify $D$ with a subset of $\overline{V}$ as in  Remark \ref{rem:matroidmod}. This is done using the identifications in Proposition  \ref{lem:Lp}.

\begin{lem} \label{lem:commutativitycomparison}
Let $\delta \colon V \rightarrow W$ be an open matroidal tropical modification  along a  matroidal divisor $D \subset W$. 
Then the 
 diagram 
\begin{align*}
\begin{xy}
\xymatrix
{
0 \ar[r] & H^{p,0}(W) \ar[r] \ar[d]^{\PD} & H^{p,0}(V) \ar[r]^{\langle \cdot \ ;e_i\rangle_p} \ar[d]^{\PD} & H^{p-1,0}(D) \ar[d]^{\PD}  \ar[r] &0 \\
0 \ar[r] & H^{n-p,n}_c(W)^* \ar[r] & H^{n-p,n}_c(V)^* \ar[r]^{g^*} & H^{n-p,n-1}_c(D)^*  \ar[r] &  0, 
}
\end{xy}
\end{align*}
which is obtained by the exact sequences in Proposition \ref{Prop:SeqComCoh} and Lemma \ref{lem:shortexactMat}, is commutative. 
\end{lem}

\begin{proof}
Note that by Proposition \ref{prop:closedmodification} the diagram is equivalent to the one obtained by replacing $W$ by $\overline{V}$. Then the fact that the first square commutes is immediate. 
The map $g \colon H^{n-p,n-1}_c(D) \rightarrow H^{n-p,n}_c(V)$ is the boundary operator in a long exact cohomology sequence. We recall its construction. 
For a closed superform $\beta \in \AS^{n-p,n-1}_c(D)$, take any lift $l(\beta) \in \AS^{n-p,n-1}_c(\overline{V})$ such that $l(\beta)|_D = \beta$. Then $d''(l(\beta))$ restricts to $0$ on $D$ and so it is  a superform with compact support on $V$. Then $g(\beta)$ is given by the class of $d''(l(\beta))$ in $H^{n-p,n}_c(V)$.
As usual this does not depend on the choice of $l(\beta)$. 
We have to show that for all closed forms $\alpha \in \AS^{p,0}(V)$, $\beta \in \AS^{n-p,n-1}_c(V)$, and a lift $l(\beta) \in  \AS^{n-p,n-1}_c(\overline{V})$ that
we have 
\begin{align*}
(-1)^p \int \limits_V \alpha \vedge d''(l(\beta)) = (-1)^{(p-1)} \int \limits_D \langle \alpha; e_i \rangle_p \vedge \beta.
\end{align*}
In the above integral  $e_i$ is the coordinate direction of the modification. 
Let $P$ be the piecewise linear function of the modification and $P' = P-1$. The graph of $P'$ divides $V$ into two subsets, one living above the graph and the other one below. Both of these subsets are the support of some polyhedral complexes, which we denote by $\CS_1$ and $\CS_2$, respectively.  Equip all facets of both polyhedral complexes $\CS_1$ and $\CS_2$ with weight one. Note that $\delta(|\CS_2|) \subset D$. We find a lift $l(\beta) \in \AS^{n-p,n-1}_c(\overline{V})$ such that $l(\beta)|_{|\CS_2|} = (\delta|_{|\CS_2|})^*(\beta)$. Then we have
\begin{align*}
\int \limits_V \alpha \vedge d''(l(\beta)) =  \int \limits_{\CS_1} \alpha \vedge d''(l(\beta)) + \int \limits_{\CS_2} \alpha \vedge d''(l(\beta)) =  \int \limits_{\CS_1} \alpha \vedge d''(l(\beta)).
\end{align*}

By Stokes' theorem in the version \cite[Proposition 3.5]{Gubler} 
\footnote{It was communicated to us by the author of \cite{Gubler} that Proposition 3.5 in \emph{loc.~cit.} contains a sign mistake. 
We use the corrected version in (\ref{Stokes boundary version}).} 
and the Leibniz rule we have
\begin{align} \label{Stokes boundary version}
\int \limits_{\CS_1} \alpha \vedge d''(l(\beta)) =  (-1)^p \int \limits_{\del \CS_1} \alpha \vedge l(\beta).
\end{align}
It follows from the proof of \cite[Theorem 3.8]{Gubler} that the boundary integral of $\alpha \vedge l(\beta)$ over balanced codimension one faces vanishes.  We further have that the unbalanced faces of $\CS_1$ are precisely the ones in the polyhedral subspace $D' := \CS_1 \cap \Gamma_{P'}(W) = \CS_2 \cap \Gamma_{P'}(W)$. 
The facets of the polyhedral subspace $D'$ are equipped with weight one. 
Thus we obtain
\begin{align*}
\int \limits_{\del \CS_1} \alpha \vedge l(\beta)  =  \int \limits_{D'} \langle \alpha \vedge l(\beta); e_i \rangle_1
= (-1)^{p-1} \int \limits_{D'} \langle \alpha \vedge l(\beta); e_i \rangle_{p}.
\end{align*}
Since $l(\beta)|_{D'}  = (\delta|_{D'})^*(\beta)$ and $\delta(e_i) = 0$, we have that $\langle l(\beta); e_i \rangle_p|_{D'} = 0$ and therefore
\begin{align*}
\int \limits_{D'} \langle \alpha \vedge l(\beta), e_i \rangle_p = 
\int \limits_{D'}  \langle \alpha; e_i \rangle_p \vedge l(\beta).
\end{align*}
Altogether, we obtain
\begin{align*}
(-1)^p \int \limits_V \alpha \vedge d''(l(\beta)) =  (-1)^{p-1} \int \limits_{D'} \langle \alpha; e_i \rangle_p \vedge l(\beta).
\end{align*}
Denote by $F \colon W \rightarrow \Gamma_{P'}(W)$ the map into the  graph of the function $P'$. Then there exist polyhedral structures $\DS$ on $D$ and $\DS'$ on $D'$, such that for each facet $\sigma$ of $\DS$ the restriction $F|_\sigma$ is linear, the image $F(\sigma)$ is a facet of $\DS'$, and each facet of $\DS'$ is of this form. Then the inverse of $F|_\sigma$ is given by $\delta|_{\sigma'}$. Thus $\delta|_{\sigma'}$ is an isomorphism of rational polyhedra for each $\sigma' \in \DS'$. Since we have that $\delta^*$ preserves $\langle \alpha; e_i \rangle_p$ and $(\delta|_{D'})^* \beta = l(\beta)|_{D'}$ we obtain
\begin{align*}
\int \limits_{D'} \langle \alpha; e_i \rangle_p \vedge l(\beta) = \int \limits_D \langle \alpha; e_i \rangle_p \vedge \beta,
\end{align*}
which concludes the proof. 
\end{proof}

\begin{Prop}\label{Prop:PDforMatroidalRr}
Let   $V \subset \R^r$ be a matroidal cycle.  Then $V$ has Poincar\'e duality. 
\end{Prop}

\begin{proof}
Let $n$ be the dimension of $V$. We perform induction on $r$. The base case $r = 0$ is obvious. \\
For the induction step, we have two cases, these being $n= r$ and $n < r$. If $n = r$, then  $V = \R^n$ and  we have 
$$H^{p,q}(\R^n) = \bigwedge^p {\R^n}^* \otimes H^q(\R^n)
 \quad  \text{and} \quad
  H_c^{p,q}(\R^n) = \bigwedge^p {\R^n}^* \otimes H_c^q(\R^n),$$ 
where $H^q$, respectively $H^q_c$, denote the usual de Rham cohomology. Therefore,  $H^{p,q}(\R^n) = 0$ and $H^{n-p,n-q}_c(\R^n) = 0$ unless $q = 0$. Otherwise $H^{p,0}(\R^n) = \bigwedge^p {\R^n}^*$, $H^{n-p,n}_c(\R^n) = \bigwedge^{n-p} {\R^n}^*$, and the $\PD$ map is just $(-1)^p$ times the map induced by the canonical pairing $\bigwedge^p {\R^n}^* \times \bigwedge^{n-p} {\R^n}^* \rightarrow \bigwedge^n {\R^n}^* \cong \R$. Since this pairing is non-degenerate the $\PD$ map is an isomorphism. 

If $n < r$, then by Proposition \ref{prop:acyclic} and Lemma \ref{lem:vanishing} the only non-trivial case to check is when $q = 0$. In other words, that $\PD \colon H^{p,0}(V) \rightarrow H^{n-p,n}_c(V)^*$ is an isomorphism. 
Consider an open matroidal tropical modification $\delta \colon V \to W$ along a matroidal  divisor $D$.
 Now $D$ and $W$ have $\PD$ by the induction hypothesis, so that in the commutative diagram from Lemma \ref{lem:commutativitycomparison} the vertical arrows on the left and right are isomorphisms. By the five lemma we obtain $\PD$ for $V \subset \R^r$ and the proposition is proven. 
\end{proof}

The next two lemmas help to prove Proposition \ref{Prop:PDforMatroidaltimesT}, which is analogous to Proposition \ref{Prop:PDforMatroidalRr} but for spaces of the form $V \times \T^r$ where $V$ is a matroidal cycle. 
We first relate the cohomologies of $V \times \T^r$, $V \times \R^r$, and $V$ by way of  an exact sequence.

\begin{lem} \label{lem:seqnoncompact}
Let $Y =  V \times \T^r$ where $V$ is the support of a polyhedral fan. Then we have a short exact sequence
\begin{align*}
\begin{xy}
\xymatrix{
0 \ar[r] & H^{p,0}(Y \times \T) \ar[r]  & H^{p,0}(Y \times \R) \ar[r]^{\langle \ \cdot \ ; e_i\rangle_p} & H^{p-1,0}(Y)  \ar[r] &0 
}
\end{xy}
\end{align*}
where 
$e_i$ is the coordinate of $\R$ in 
$Y \times \R$.
\end{lem}
\begin{proof}
We use the explicit calculation in Proposition  \ref{lem:Lp}. First this shows that none of the cohomology groups in the statement change when we replace $Y$ by $V$, thus we assume $Y = V$. 
We begin by showing that the sequence 
\begin{align} \label{sequencesum}
\begin{xy}
\xymatrix{
0 \ar[r] & \sum \limits_{\sigma \in Y} \bigwedge^{p-1} \Linear(\sigma) \ar[r]^{\vedge e_i} &\sum \limits_{\sigma \in Y} \bigwedge^p \Linear(\sigma \times \R) \ar[r] & \sum \limits_{\sigma \in Y} \bigwedge^p  \Linear(\sigma)  \ar[r] & 0	
}
\end{xy}		
\end{align}
is exact. The first map is clearly injective and the last map is clearly surjective.  For exactness in the middle notice the composition of the maps is certainly zero and that any element $v \in \sum \limits_{\sigma \in Y} \bigwedge^p \Linear(\sigma \times \R)$ can be written as $v = \sum \limits_{\sigma \in Y} v_\sigma + \left( \sum \limits_{\sigma \in Y} v'_\sigma \right) \vedge e_i$ for $v_\sigma \in \bigwedge^p \Linear(\sigma)$ and $v'_\sigma \in \bigwedge^{p-1} \Linear(\sigma)$. If $v$ maps to zero then it is of the form $\left( \sum \limits_{\sigma \in Y} v'_\sigma \right)\wedge e_i$
 and thus in the image of $\vedge e_i$. This proves exactness of that sequence.  
 
Identifying $\LS^p$ with $H^{p,0}$ and dualizing the sequence (\ref{sequencesum}) completes the proof. 
\end{proof}

\begin{lem}\label{lem:commuteProd} 
Let $Y = V \times \T^r $ be pure $n$-dimensional, where $V \subset \R^s$ is a matroidal cycle, then the following diagram
\begin{align*}
\begin{xy}
\xymatrix
{
 H^{p,0}(Y \times \T) \ar[r] \ar[d]^{\PD} & H^{p,0}(Y \times \R) \ar[r]^{\langle \ \cdot \  ; e_i\rangle_p} \ar[d]^{\PD} & H^{p-1,0}(Y) \ar[d]^{\PD}  \\
 H^{n-p+1,n+1}_c(Y \times \T)^* \ar[r] & H^{n-p+1,n+1}_c(Y \times \R)^* \ar[r]^{ \ \ \ \  g^*} & H^{n-p+1,n}_c(Y)^*,
}
\end{xy}
\end{align*}
which is obtained by the sequences in Proposition \ref{Prop:SeqComCoh} and Lemma \ref{lem:seqnoncompact}, commutes.
\end{lem}

\begin{proof}
The proof follows exactly along the lines of the proof of the commutativity of the diagram in Lemma \ref{lem:commutativitycomparison} for tropical modifications with $Y$ replacing $D$, $Y \times \R$ replacing $V$, $Y \times \T$ replacing $\overline{V}$ and $P'$ being any constant function. 
\end{proof}

\begin{Prop} \label{Prop:PDforMatroidaltimesT}
Let $Y = V \times\T^r$ for a matroidal cycle $V \subset \R^s$. Then $Y$ has Poincar\'e duality.
\end{Prop}
\begin{proof}
We perform an induction on $r$. The base case $r = 0$ follows from Proposition \ref{Prop:PDforMatroidalRr}. For the induction step we have to show that if $Y$ has $\PD$ then $Y \times \T$ also has $\PD$. Since $Y$ is a basic open subset, $H^{p,q}(Y) = 0$ unless $q = 0$ by Proposition \ref{prop:acyclic}. Since $Y$ has $\PD$, we have  $H^{p,q}_c(Y) = 0$ unless $q = n = \dim(Y)$.  Note also that $Y \times \R = V \times \R \times \T^r$ and so this space has $\PD$. Therefore  $H^{p,q}_c(Y \times \R) = 0$ unless $q = n + 1 = \dim(Y \times \R)$.
Then the sequence from Proposition \ref{Prop:SeqComCoh} yields that $H^{p,q}_c(Y \times \T) = 0$ if $q \neq n, n+1$ and that the sequence 
\begin{align}\label{seq:longcompact}
0 \rightarrow H^{p,n}_c(Y \times \T) \rightarrow H^{p,n}_c(Y) \overset{f}{\rightarrow} H^{p,n+1}_c(Y \times \R) \rightarrow H^{p,n+1}_c(Y \times \T) \rightarrow 0 
\end{align}
is exact. By the commutativity of the second square of the diagram in Lemma  \ref{lem:commuteProd}, 
the map  $f$ in (\ref{seq:longcompact}) is dual to the map $\langle \ \cdot \ ;  e_i \rangle_{n-p+1}$. 
  This can be seen once we apply $\PD$ for $Y$ and $Y \times \R$ to 
  identify $H^{p,0}(Y \times \R) \cong H^{n-p+1,n+1}_c(Y \times \R)^*$ and $H^{p-1,0}(Y) \cong H^{n-p+1,n}_c(Y)^*$. Now $\langle  \ \cdot \ ;  e_i \rangle_{n-p+1}$ is  surjective by Lemma \ref{lem:seqnoncompact}, thus $f$ is injective and we have $H^{p,q}_c(Y \times \T) = 0$ unless $q = n+1$. 
  
Since $Y \times \T$ is a basic open subset, we also have $H^{p,q}(Y \times \T) = 0$ unless $q = 0$ by Proposition \ref{prop:acyclic}. Therefore, it remains to consider $\PD \colon H^{p,0}(Y \times \T) \rightarrow H^{n+1-p, n+1}_c(Y \times \T)$. Note that this is precisely the first vertical map in the diagram in Lemma \ref{lem:commuteProd}. The respective first horizontal maps are injective by Lemma \ref{lem:seqnoncompact} and the sequence (\ref{seq:longcompact}). Since the other vertical maps are isomorphisms, this shows that $Y\times \T$ has $\PD$. 
\end{proof}

The following technical lemma allows us to deduce Poincar\'e duality for basic open subsets
 of  matroidal fans.

\begin{lem} \label{lem:starfan}
Let $V \subset \R^s$ be a matroidal cycle,  $Y = V \times \T^r$, and 
$\Omega$ a basic open neighborhood of  $(0,\dots,0,-\infty,\dots,-\infty)$. Then there are canonical isomorphisms
$$H^{p,q}(Y) 
\cong H^{p,q}(\Omega) \qquad \text{and}  \qquad  
H^{p,q}_c(\Omega)  
\cong H^{p,q}_c(Y)$$
which are induced by restriction and inclusion of superforms. In particular $\Omega$ has $\PD$.
\end{lem}
\begin{proof}
For $H^{p,q}$ the statement follows immediately from the explicit calculation done in Proposition \ref{lem:Lp}. For cohomology with compact support,  we first see that there is a homeomorphism between $\Omega$ and $Y$ which  respects strata of $\T^r$ and the polyhedra for any fan polyhedral structure on $Y$.
This homeomorphism  induces an isomorphism of tropical cohomology with compact support and therefore  $H^{p,q}_c(\Omega) \cong H^{p,q}_c(Y)$ by Theorem \ref{thm:equivalenceDeRhamPQ}. By $\PD$ for $Y$ these cohomology groups with compact support are finite dimensional, thus it is sufficient to show that inclusion of superforms induces a surjective map on cohomology. Again by $\PD$ for $Y$ this is trivial if $q \neq n = \dim Y$. 

For $q = n$,  choose a basis $\alpha_1,\dots,\alpha_k$ of $H^{p,0}(Y)$. By $\PD$ for $Y$ there exist $\omega_1,\dots,\omega_k \in H^{n-p,n}_c(Y)$ such that $\int_Y \alpha_i \vedge \omega_j = \delta_{ij}$. For surjectivity of $H^{n-p,n}_c(\Omega) \rightarrow H^{n-p,n}_c(Y)$ it is sufficient to show that there exist $\beta_1,\dots,\beta_k \in H^{n-p,n}_c(\Omega)$ such that $\int_Y \alpha_i \vedge \beta_j \neq 0$ if and only if $i = j$. Let $B$ be the union of the supports of all $\omega_i$. Take  $C \in \R_{>0}$ and $v \in \R^{r}$ such that $B \subset C \cdot \Omega + v$. Define $F$ to be the extended affine map given by $w \mapsto C \cdot w + v$ and set $\beta_i := F^*(\omega_i) \in \AS^{p,q}_c(\Omega)$ for all $i$. Since $\alpha_i \in \LS^p(Y)$ we have $F^*(\alpha_i) = C^p \alpha_i$ and 
\begin{align*}
\int\limits_Y \alpha_i \vedge \beta_j = \int \limits_Y \alpha_i \vedge F^*(\omega_j) = C^{-p} \int \limits_Y F^*(\alpha_i \vedge \omega_j) = C^{n-p}\delta_{ij},
\end{align*}
where the last equality is given by  the transformation formula \cite[2.4]{Gubler}. This proves surjectivity and thus the lemma. 
\end{proof}

\begin{lem}\label{lem:matroidalbasicopen}
Let $V$ be a matroidal cycle in $\R^s$,  $Y = V \times \T^r$, and 
$\Omega$ a basic open subset for some polyhedral structure on $Y$. Then $\Omega$ has $\PD$. 
\end{lem}
\begin{proof}
If $I \subset [r]$ is the maximal sedentarity among points of $\Omega$,
then  
$\Omega$ is a basic open subset of a polyhedral structure on  $ V \times \TT^{|I|} \times \R^{r-|I|}$ of maximal sedentarity. 
Let $x$ be a point in the relative interior of the minimal face of the basic open set $\Omega$. 
The star of any point in a matroidal fan is again a matroidal fan, see \cite[Proposition 2]{ArdilaKlivans}. 
Applying this fact to $V \times \R^{r- |I|}$ we obtain that, after translation of $x$ to the origin, the basic open set $\Omega$ is a  neighborhood of $(0,\dots,0,-\infty,\dots,-\infty)$ in the matroidal cycle $|\text{Star}_x(\Sigma)| \times \T^{|I|}$ for a matroidal fan $\Sigma$ with support $V$.
Therefore, we can apply Lemma \ref{lem:starfan}, and it follows that $\Omega$ has PD. 
\end{proof}

Finally, we obtain Poincar\'e duality for tropical manifolds:

\begin{Thm}\label{Thm:PoincareDualityII}
Let $X$ be an $n$-dimensional tropical manifold. Then 
the Poincar\'e duality map 
is an isomorphism for all $p, q$.
\end{Thm}

\begin{proof}
We write ${\AS^{p,q}_c}^*$ for the sheaf $U \mapsto \Hom_\R(\AS^{p,q}_c(U), \R)$. Then ${\AS^{p, q}_c}^*$ is a sheaf, since $\AS^{p,q}$ is fine. Furthermore ${\AS^{p,q}_c}^*$ is a flasque sheaf, since for $U' \subset U$ the inclusion $A^{p,q}_c(U') \rightarrow A^{p,q}_c(U)$ is injective. We then obtain the commutative diagram
\begin{align*}
\begin{xy}
\xymatrix
{
0 \ar[r] & \LS^p \ar[r] \ar[d]^{\id} & \AS^{p,0} \ar[r]^{d''} \ar[d]^{\PD} & \AS^{p,1}\ar[r] \ar[d]^{\PD} &\dots \\
0 \ar[r] &\LS^p \ar[r]& {\AS^{n-p,n}_{c}}^* \ar[r]^{{d''}^*}& {\AS^{n-p,n-1}_{c}}^* \ar[r]& \dots
}
\end{xy}
\end{align*}
and we have 
\begin{align*}
H^q({\AS^{n-p,n-\bullet}_c}^*(U), {d''}^*) = (H_q(\AS^{n-p,n-\bullet}_c(U), d''))^* = H^{p,q}_c(U)^*.
\end{align*}
If we consider the sections of this diagram over a basic open subset, then the first row is exact by Proposition \ref{prop:acyclic}. By Lemma \ref{lem:matroidalbasicopen} the second row is  also exact. This shows that both rows  are exact sequences of sheaves on $X$. Thus we have a commutative diagram of acyclic resolutions of $\LS^p$, thus $\PD$ induces isomorphisms on the cohomology of the complexes of global sections. This precisely means that $X$ has $\PD$. 
\end{proof}

When $X$ is a compact tropical manifold,
 the above theorem immediately implies the following. 
\begin{kor} \label{kor:PD}
Let $X$ be a compact tropical manifold of dimension $n$. Then
$$\PD \colon H^{p,q}(X)  \to H^{n-p,n-q}(X)^*$$
is an isomorphism for all $p,q$.
\end{kor}

\bibliographystyle{alpha}

\def\cprime{$'$}

\end{document}